\documentclass[final]{article}
\usepackage{fullpage}

\usepackage{amsmath,amssymb,graphicx,epsfig,color,mathbbol}
\usepackage{graphics, graphicx}
\usepackage{enumerate}
\usepackage{subfigure}
\usepackage{verbatim}
\usepackage{bm} % for bold greek symbols

\usepackage{enumerate}
\usepackage{palatino}

\usepackage{hyperref}
\hypersetup{
    unicode=false,          % non-Latin characters in Acrobat’s bookmarks
    pdftoolbar=true,        % show Acrobat’s toolbar?
    pdfmenubar=true,        % show Acrobat’s menu?
    pdffitwindow=true,      % page fit to window when opened
    pdftitle={Randomized Extended Kaczmarz for Solving Least Squares},    % title
    pdfauthor={Anastasios Zouzias},     % author
    pdfsubject={Least squares, Kaczmarz, randomized algorithms, iterative algorithms, numerical linear algebra},   % subject of the document
    pdfcreator={Anastasios Zouzias},   % creator of the document
    pdfproducer={}, % producer of the document
    pdfkeywords={}, % list of keywords
    pdfnewwindow=true,      % links in new window
    colorlinks=false,       % false: boxed links; true: colored links
    linkcolor=red,          % color of internal links
    citecolor=green,        % color of links to bibliography
    filecolor=magenta,      % color of file links
    urlcolor=cyan          % color of external links
}

\newcommand{\hide}[1]{}
\newcommand\remove[1]{}

\newcommand{\E}{\mathbb{E}}

\newcommand{\x}{\mathbf{x}}
\newcommand{\y}{\mathbf{y}}
\newcommand{\z}{\mathbf{z}}

\newcommand{\zero}{\mathbf{0}}

\newcommand{\eps}{\varepsilon}

\newcommand{\Prob}[1]{\ensuremath{\mathbb{P}\left(#1\right)}}

\newcommand{\RR}{\mathbb{R}}
\newcommand{\NN}{\mathbb{N}}

\DeclareMathOperator{\EE}{\mathbb{E}}
\newcommand{\pinv}[1]{ {#1}^\dagger}
\newcommand{\norm}[1]{\ensuremath{\left\|#1\right\|_2}}

\newcommand{\frobnorm}[1]{\ensuremath{\left\|#1\right\|_{\text{\rm F}}}}
\newcommand{\rank}[1]{\ensuremath{\mathrm{\textbf{{\footnotesize rank}}}\left(#1\right)}}
\newcommand{\cond}[1]{\ensuremath{\kappa^2\left(#1\right)}}

\newcommand{\nnz}[1]{\ensuremath{\mathrm{\textbf{\footnotesize nnz}}\left(#1\right)}}

\newcommand{\ravg}{\text{R}_{\text{avg}}}
\newcommand{\cavg}{\text{C}_{\text{avg}}}

\newcommand{\kappaFS}{\kappa^2_{\textrm{\tiny F}}}
\newcommand{\kappaF}{\kappa_{\textrm{\tiny F}}}

\newcommand{\ip}[2]{\left\langle {#1},\ {#2} \right\rangle}

\newcommand{\ignore}[1]{}

\newcommand{\Id}{\mathbf{I}}

\newcommand{\zeromtx}{\mathbf{0}}

\newcommand{\mat}[1]{ {\ensuremath{\mathsf{#1} }}}
\newcommand{\matA}{\mat{A}}

\newcommand{\matP}{\mat{P}}

\newcommand{\e}{\ensuremath{{\mathbf e}}}

\newcommand{\w}{{\mathbf{w}}}
\newcommand{\vecu}{\mathbf{u}}
\newcommand{\vecv}{\mathbf{v}}
\def\b{{\mathbf b}}
\newcommand{\bc}{{\b_{\mathcal{R}(\matA)^\bot } }}
\newcommand{\br}{{\b_{\mathcal{R}(\matA) } }}

\def\xls{\x_{\text{\tiny LS}}}

\newcommand{\ar}[1]{ \matA^{(#1)}}
\newcommand{\ac}[1]{ \matA_{(#1)}}

\newcommand{\colspan}[1]{\mathcal{R}(#1)}

\usepackage{algorithmicx}
\usepackage[ruled]{algorithm}
\usepackage{algpseudocode}

\usepackage{amsthm}
\newtheorem{theorem}{Theorem}
\newtheorem{lemma}[theorem]{Lemma}

\newtheorem{fact}[theorem]{Fact}

\newtheorem{remark}{Remark}

\title{Randomized Extended Kaczmarz for Solving Least Squares}

\author{Anastasios Zouzias%
\thanks{Anastasios Zouzias is with the Department of Computer Science at the University of Toronto, Canada. E-mail: {\tt zouzias@cs.toronto.edu}. Part of this work was done while the author was visiting the Department of Computer Science at Princeton University.} 
\and%
Nikolaos M. Freris\thanks{Nikolaos M. Freris is with IBM Research - Z\"{u}rich, S\"{a}umerstrasse 4, 8803 R\"{u}schlikon, Switzerland. E-mail: {\tt nif@zurich.ibm.com} }}

\begin{document}

\maketitle

%%%%%%%%%%%%%%%%%%%%%%%%%%%%%%%%%%%%%%%%%%%%%%%%%%
%%%%%%%%%%%%%%%%%%%%%%%%%%%%%%%%%%%%%%%%%%%%%%%%%%
\begin{abstract}
We present a randomized iterative algorithm that exponentially converges in expectation to the minimum Euclidean norm least squares solution of a given linear system of equations. The expected number of arithmetic operations required to obtain an estimate of given accuracy is proportional to the square condition number of the system multiplied by the number of non-zeros entries of the input matrix. The proposed algorithm is an extension of the randomized Kaczmarz method that was analyzed by Strohmer and Vershynin.
\end{abstract}
%%%%%%%%%%%%%%%%%%%%%%%%%%%%%%%%%%%%%%%%%%%%%%%%%%%
%%%%%%%%%%%%%%%%%%%%%%%%%%%%%%%%%%%%%%%%%%%%%%%%%%%

%\begin{keywords} 
%least squares, iterative methods, randomized algorithms
%\end{keywords}

%\begin{AMS}
%XXXXX, XXXXX, XXXXX
%\end{AMS}

%\pagestyle{myheadings}
%\thispagestyle{plain}
%\markboth{A. ZOUZIAS AND N. FRERIS}{RANDOMIZED EXTENDED KACZMARZ FOR SOLVING LEAST SQUARES}

%%%%%%%%%%%%%%%%%%%%%%%%%%%%%%%%%%%%%%%%%%%%%%%%%%%
%%%%%%%%%%%%%%%%%%%%%%%%%%%%%%%%%%%%%%%%%%%%%%%%%%%
\section{Introduction}
%%%%%%%%%%%%%%%%%%%%%%%%%%%%%%%%%%%%%%%%%%%%%%%%%%%
%%%%%%%%%%%%%%%%%%%%%%%%%%%%%%%%%%%%%%%%%%%%%%%%%%%
%We assume without loss of generality that all rows and columns of $\matA$ are non-zero.
%

%
The Kaczmarz method is an iterative projection algorithm for solving linear systems of equations~\cite{K}. Due to its simplicity, the Kaczmarz method has found numerous applications including image reconstruction, distributed computation and signal processing to name a few~\cite{K:apps:CFM92,book:K:apps:H80,book:K:apps:Nat01,FZ12}, see~\cite{K:apps} for more applications. The Kaczmarz method has also been rediscovered in the field of image reconstruction and called ART (Algebraic Reconstruction Technique)~\cite{ART}, see also~\cite{book:Zenios,book:K:apps:H80} for additional references. It has been also applied to more general settings, see~\cite[Table~1]{K:apps} and \cite{K:tompk,K:rate:MC77} for non-linear versions of the Kaczmarz method. 

Let $\matA\in\RR^{m\times n}$ and $\b\in\RR^m$. Throughout the paper all vectors are assumed to be column vectors. The Kaczmarz method operates as follows: Initially, it starts with an arbitrary vector $\x^{(0)}\in\RR^n$. In each iteration, the Kaczmarz method goes through the rows of $\matA$ in a cyclic manner\footnote{That is, selecting the indices of the rows from the sequence $1,2,\ldots , m , 1,2 , \ldots$.} and for each selected row, say $i$-th row $\ar{i}$, it orthogonally projects the current estimate vector onto the affine hyperplane defined by the $i$-th constraint of $\matA\x = \b$, i.e., $\{\x\ |\ \ip{\ar{i}}{\x} = b_i\}$ where $\ip{\cdot}{\cdot}$  is the Euclidean inner product. More precisely, assuming that the $i_k$-th row has been selected at $k$-th iteration, then the $(k+1)$-th estimate vector $\x^{(k+1)}$ is inductively defined by
\[\x^{(k+1)} := \x^{(k)} + \lambda_k\frac{b_{i_k} - \ip{\ar{i_k}}{ \x^{(k)}}}{\norm{\ar{i_k}}^2} \ar{i_k}\]
where $\lambda_k \in \RR$ are the so-called relaxation parameters and $\norm{\cdot}$ denotes the Euclidean norm. The original Kaczmarz method corresponds to $\lambda_k = 1$ for all $k\geq 0$ and all other setting of $\lambda_k$'s are usually referred as the \emph{relaxed Kaczmarz method} in the literature~\cite{K:apps,book:Galantai}.

Kaczmarz proved that this process converges to the unique solution for square non-singular matrices~\cite{K}, but without any attempt to bound the rate of convergence. Bounds on the rate of convergence of the Kaczmarz method are given in~\cite{K:rate:MC77}, \cite{K:rate:Ansorge} and \cite[Theorem~4.4, p.120]{book:Galantai}. In addition, an error analysis of the Kaczmarz method under the finite precision model of computation is given in~\cite{phdthesis:K:error,K:error}.
%

%
%More precisely, let $\widehat{\matA}\in\RR^{m\times n}$ whose row set is the rows of $\matA$ normalized to unit length. Then it follows that after a sweep over all rows of $\matA$, Kaczmarz's method improves its estimate by $1- \det (\widehat{\matA}^\top \widehat{\matA})$.

Nevertheless, the Kaczmarz method converges even if the linear system $\matA\x = \b$ is overdetermined ($m>n$) and has no solution. In this case and provided that $\matA$ has full column rank, the Kaczmarz method converges to the least squares estimate. This was first observed by Whitney and Meany~\cite{K:rate:WM67} who proved that the relaxed Kaczmarz method converges provided that the relaxation parameters are within $[0,2]$ and $\lambda_k\to 0$, see also~\cite[Theorem~1]{K:relax:CEG83}, \cite{K:Tanabe} and~\cite{K:relax:Hanke90} for additional references.
%STATE the result here. (THEOREM 4.32 from BOOK, see also Herman, Lent, Lutz and CEG82). where and later Tanabe proved it~\cite{K:Tanabe}.

In the literature there was empirical evidence that selecting the rows non-uniformly at random may be more effective than selecting the rows via Kaczmarz's cyclic manner~\cite{RK:HM93,K:apps:CFM92}. Towards explaining such an empirical evidence, Strohmer and Vershynin proposed a simple randomized variant of the Kaczmarz method that has exponential convergence \emph{in expectation}~\cite{RK} assuming that the linear system is solvable; see also~\cite{LS:RCD} for extensions to linear constraints. A randomized iterative algorithm that computes a sequence of random vectors $\x^{(0)}, \x^{(1)}, \ldots$ is said to \emph{converge in expectation} to a vector $\x^*$ if and only if $\EE \norm{\x^{(k)} - \x^*}^2\to 0$ as $k\to\infty$, where the expectation is taken over the random choices of the algorithm. Soon after~\cite{RK}, Needell analyzed the behavior of the randomized Kaczmarz method for the case of full column rank linear systems that do not have any solution~\cite{Needell09}. Namely, Needell proved that the randomized Kaczmarz estimate vector is (in the limit) within a fixed distance from the least squares solution and also that this distance is proportional to the distance of $\b$ from the column space of $\matA$. In other words, Needell proved that the randomized Kaczmarz method is effective for least squares problems whose least squares error is negligible.

In this paper we present a randomized iterative least squares solver (Algorithm~\ref{alg:REK}) that converges in expectation to the minimum Euclidean norm solution of \begin{equation}\label{eq:ls}
\min_{\x\in\RR^n}\norm{\matA \x - \b}.
\end{equation}
The proposed algorithm is based on~\cite{RK,Needell09} and inspired by~\cite{popa}. More precisely the proposed algorithm can be thought of as a randomized variant of Popa's extended Kaczmarz method~\cite{popa}, therefore we named it as \emph{randomized extended Kaczmarz}.

%In this setting, it turns out that the randomized Kaczmarz estimate vector is within a fixed distance from the least squares solution. In addition, the distance is proportional, roughly speaking,  to the distance of $\b$ from the column space of $\matA$ which in general can be arbitrarily large.
%
%
% We mostly follow this book~\cite{book:Galantai}. 
%Books : \cite{book:Bodewig}, \cite{book:Gastinel}

%
\paragraph{Organization of the paper}
In Section~\ref{sec:related}, we briefly discuss related work on the design of deterministic and randomized algorithms for solving least squares problems. In Section~\ref{sec:back}, we present a randomized iterative algorithm for projecting a vector onto a subspace (represented as the column space of a given matrix) which may be of independent interest. In addition, we discuss the convergence properties of the randomized Kaczmarz algorithm for solvable systems (Section~\ref{sec:RK}) and recall its analysis for non-solvable systems (Section~\ref{sec:noisyRK}). In Section~\ref{sec:result}, we present and analyze the randomized extended Kaczmarz algorithm. Finally, in Section~\ref{sec:impl} we provide a numerical evaluation of the proposed algorithm.
%
%
%%%%%%%%%%%%%%%%%%%%%%%%%%%%%%%%%%%%%%%%%%%%%%%%%%%%%%%
%%%%%%%%%%%%%%%%%%%%%%%%%%%%%%%%%%%%%%%%%%%%%%%%%%%%%%%
\section{Least squares solvers}\label{sec:related}
%%%%%%%%%%%%%%%%%%%%%%%%%%%%%%%%%%%%%%%%%%%%%%%%%%%%%%%
%%%%%%%%%%%%%%%%%%%%%%%%%%%%%%%%%%%%%%%%%%%%%%%%%%%%%%%
%
In this section we give a brief discussion on least squares solvers including deterministic direct and iterative algorithms together with recently proposed randomized algorithms. For a detailed discussion on deterministic methods, the reader is referred to~\cite{book:Bjork}. In addition, we place our contribution in context with prior work.
\paragraph{Deterministic algorithms}
In the literature, several methods have been proposed for solving least squares problems of the form~\eqref{eq:ls}. Here we briefly describe a representative sample of such methods including the use of QR factorization with pivoting, the use of the singular value decomposition (SVD) and iterative methods such as Krylov subspace methods applied on the normal equations~\cite{book:saad}. LAPACK provides robust implementations of the first two methods; DGELSY uses QR factorization with pivoting and DGELSD uses the singular value decomposition~\cite{LAPACK}. For the iterative methods, LSQR is equivalent to applying the conjugate gradient method on the normal equations~\cite{PS82} and it is a robust and numerically stable method.
\paragraph{Randomized algorithms}
To the best of our knowledge, most randomized algorithms proposed in the theoretical computer science literature for approximately solving least squares are mainly based on the following generic two step procedure: first randomly (and efficiently) project the linear system into sufficiently many dimensions, and second return the solution of the down-sampled linear system as an approximation to the original optimal solution~\cite{petrosSODA06,Sarlos,CW09,Nguyen09,MZ11,fasterLS}, see also~\cite{ls:nnzA}. Concentration of measure arguments imply that the optimal solution of the down-sampled system is close to the optimal solution of the original system. The accuracy of the approximate solution using this approach depends on the sample size and to achieve relative accuracy $\eps$, the sample size should depend inverse polynomially on $\eps$. This makes these approaches unsuitable for the high-precision regime of error that is considered here.

A different approach is the so called randomized preconditioning method, see~\cite{RT08,AMT10}. The authors of~\cite{AMT10} implemented Blendenpik, a high-precision least squares solver. Blendenpik consists of two steps. In the first step, the input matrix is randomly projected and an effective preconditioning matrix is extracted from the projected matrix. In the second step, an iterative least squares solver such as the LSQR algorithm of Paige and Saunders~\cite{PS82} is applied on the preconditioned system. Blendenpik is effective for overdetermined and underdetermined problems.

A parallel iterative least squares solver based on normal random projections called LSRN was recently implemented by Meng, Saunders and Mahoney~\cite{lsrn}. LSRN consists of two phases. In the first preconditioning phase, the original system is projected using random normal projection from which a preconditioner is extracted. In the second step, an iterative method such as LSQR or the Chebyshev semi-iterative method~\cite{Chebyshev} is applied on the preconditioned system. This approach is also effective for over-determined and under-determined least squares problems assuming the existence of a parallel computational environment.
%
%%%%%%%%%%%%%%%%%%%%%%%%%%%%%%%%%%%%%%%%%%%%%%%%%%%%%%%
%%%%%%%%%%%%%%%%%%%%%%%%%%%%%%%%%%%%%%%%%%%%%%%%%%%%%%%
\subsection{Relation with our contribution}
%%%%%%%%%%%%%%%%%%%%%%%%%%%%%%%%%%%%%%%%%%%%%%%%%%%%%%%
%%%%%%%%%%%%%%%%%%%%%%%%%%%%%%%%%%%%%%%%%%%%%%%%%%%%%%%
%
In Section~\ref{sec:impl}, we compare the randomized extended Kaczmarz algorithm against DGELSY, DGELSD, Blendenpik. LSRN~\cite{lsrn} did not perform well under a setup in which no parallelization is allowed, so we do not include LSRN's performance. The numerical evaluation of Section~\ref{sec:impl} indicates that the randomized extended Kaczmarz 
is effective on the case of sparse, well-conditioned and strongly rectangular (both overdetermined and underdetermined) least squares problems, see~Figure~\ref{fig:sparse}. Moreover, the randomized extended Kaczmarz algorithm has also comparable performance with LAPACK's routine for the dense random input matrices, see~Figure~\ref{fig:dense} (notice that the proposed algorithm almost matches Blendenpik's performance for the underdetermined case, see Figure~\ref{fig:denseUnder}). On the other hand, a preconditioned version of the proposed algorithm does not perform well under the case of ill-conditioned matrices, see Figure~\ref{fig:cond}.
%
%
%%%%%%%%%%%%%%%%%%%%%%%%%%%%%%%%%%%%%%%%%%%%%%%%%%%%%%%
%%%%%%%%%%%%%%%%%%%%%%%%%%%%%%%%%%%%%%%%%%%%%%%%%%%%%%%
\section{Background}\label{sec:back}
%%%%%%%%%%%%%%%%%%%%%%%%%%%%%%%%%%%%%%%%%%%%%%%%%%%%%%%
%%%%%%%%%%%%%%%%%%%%%%%%%%%%%%%%%%%%%%%%%%%%%%%%%%%%%%%
%
%
\paragraph{Preliminaries and Notation} For an integer $m$, let $[m]:=\{1,\ldots,m\}$. Throughout the paper all vectors are assumed to be column vectors. We denote the rows and columns of $\matA$ by $\ar{1}, \ldots , \ar{m}$ and $\ac{1},\ldots , \ac{n}$, respectively (both viewed as column vectors). $\colspan{\matA}$ denotes the column space of $\matA$, i.e., $\colspan{\matA}:=\{\matA \x\ | \ \x\in\RR^n\}$ and $\colspan{\matA}^{\bot}$ denotes the orthogonal complement of $\colspan{\matA}$. Given any $\b\in\RR^m$, we can uniquely write it as $\br + \bc$, where $\br$ is the projection of $\b$ onto $\colspan{\matA}$. $\frobnorm{\matA}:=\sqrt{\sum_{i=1}^{m}\sum_{j=1}^{n}|a_{ij}|^2}$ and $\norm{\matA}:=\max_{\x \neq \zeromtx } \norm{\matA \x }/ \norm{\x}$ denotes the Frobenius norm and spectral norm, respectively. Let $\sigma_1\geq\sigma_2\geq\ldots \geq\sigma_{\rank{\matA}}$ be the non-zero singular values of $\matA$. We will usually refer to $\sigma_1$ and $\sigma_{\rank{\matA}}$ as $\sigma_{\max}$ and $\sigma_{\min}$, respectively. The Moore-Pensore pseudo-inverse of $\matA$ is denoted by $\pinv{\matA}$~\cite{book:GVL}. Recall that $\norm{\pinv{\matA}} = 1/\sigma_{\min}$.  
%We denote the inner product between two row or column vectors $\x$ and $\y$ of the same dimensions by $\ip{\x}{\y}:=\sum_{i} x_i y_i$.
For any non-zero real matrix $\matA$, we define
%%%%%%%%%%%%%%%%%%%%%%%%%%%%%%%%%%%%%%%%%%%%%%%%%%%%%%%
\begin{equation}\label{eq:kappa}
\kappaFS(\matA) := \frobnorm{\matA}^2 \norm{\pinv{\matA}}^2.
\end{equation}
%%%%%%%%%%%%%%%%%%%%%%%%%%%%%%%%%%%%%%%%%%%%%%%%%%%%%%%
Related to this is the scaled square condition number introduced by Demmel in~\cite{cond:Demmel}, see also~\cite{RK}. It is easy to check that the above parameter $\kappaFS(\matA)$ is related with the condition number of $\matA$, $\cond{\matA}:= \sigma^2_{\max} / \sigma^2_{\min}$, via the inequalities: $\cond{\matA} \leq \kappaFS(\matA) \leq \rank{\matA} \cdot \cond{\matA}$. We denote by $\nnz{\cdot}$ the number of non-zero entries of its argument matrix. We define the \emph{average row sparsity} and \emph{average column sparsity} of $\matA$ by $\ravg$ and $\cavg$, respectively, as follows:
\[	\ravg := \sum_{i=1}^{m} q_i \nnz{\ar{i}}\quad\text{and}\quad \cavg := \sum_{j=1}^{n} p_{j} \nnz{\ac{j}}\]
where $p_j := \norm{\ac{j}}^2 / \frobnorm{\matA}^2$ for every $i\in{[n]}$ and $q_i := \norm{\ar{i}}^2 / \frobnorm{\matA}^2$ for every $i\in{[m]}$. The following fact will be used extensively in the paper.
%
%%%%%%%%%%%%%%%%%%%%%%%%%%%%%%%%%%%%%%%%%%%%%%%%%%%%%%%
\begin{fact}\label{fact:xls}
Let $\matA$ be any non-zero real $m\times n$ matrix and $\b\in\RR^m$. Denote by $\xls:= \pinv{\matA}\b$. Then $\xls = \pinv{\matA}\br$.
\end{fact}
%%%%%%%%%%%%%%%%%%%%%%%%%%%%%%%%%%%%%%%%%%%%%%%%%%%%%%%
%
We frequently use the inequality $1-t\leq \exp(-t)$ for every $t\leq 1$. We conclude this section by collecting a few basic facts from probability theory that will be frequently used. For any random variable $X$, we denote its expectation by $\EE[X]$ or $\EE X$. If $X$ is a non-negative random variable, Markov's inequality states that $\Prob{ X > t } \leq t^{-1}\EE [X]$. Let $X$ and $Y$ be two random variables, then $\EE[X+Y]=\EE[X] + \EE[Y]$. We will refer to this fact as \emph{linearity of expectation}. Let $\mathcal{E}_1, \mathcal{E}_2, \ldots , \mathcal{E}_l$ be a set of events defined over some probability space holding with probabilities $p_1,p_2,\ldots p_l$ respectively, then $\Prob{\mathcal{E}_1\cup \mathcal{E}_2\cup \ldots \cup \mathcal{E}_l} \leq \sum_{i=1}^{l}p_i$. We refer to this fact as \emph{union bound}.
%
%%%%%%%%%%%%%%%%%%%%%%%%%%%%%%%%%%%%%%%%%%%%%%%%%%%%%%%
%%%%%%%%%%%%%%%%%%%%%%%%%%%%%%%%%%%%%%%%%%%%%%%%%%%%%%%
\subsection{Randomized Approximate Orthogonal Projection}
%%%%%%%%%%%%%%%%%%%%%%%%%%%%%%%%%%%%%%%%%%%%%%%%%%%%%%%
%%%%%%%%%%%%%%%%%%%%%%%%%%%%%%%%%%%%%%%%%%%%%%%%%%%%%%%
%
%
\begin{algorithm}{}
	\caption{Randomized Orthogonal Projection}\label{alg:randOP}
\begin{algorithmic}[1]
\Procedure{}{$\matA$, $\b$, $T$}\Comment{$\matA\in\RR^{m\times n}, \b\in\RR^m$, $T\in \NN$}
\State Initialize $\z^{(0)} =\b$ 
\For {$k=0,1,2,\ldots, T - 1 $ }
	\State Pick $j_k\in[n]$ with probability $p_j:=\norm{\ac{j}}^2/\frobnorm{\matA}^2,\ j\in [n]$
	\State Set $ \z^{(k+1)} = \left(\Id_m - \frac{\ac{j_k} \ac{j_k}^\top }{\norm{\ac{j_k}}^2}\right) \z^{(k)}$
\EndFor
\State Output $\z^{(T)}$
\EndProcedure
\end{algorithmic}
\end{algorithm}
In this section we present a randomized iterative algorithm (Algorithm~\ref{alg:randOP}) that, given any vector $\b\in\RR^m$ and a linear subspace of $\RR^m$ represented as the column space of a given matrix $\matA$, approximately computes the orthogonal projection of $\b$ onto the column space of $\matA$ (denoted by $\br$, $\br=\matA\pinv{\matA}\b$), see~\cite{ROP:CRT11} for a different approach.
Algorithm~\ref{alg:randOP} is iterative. Initially, it starts with $\z^{(0)}=\b$. At the $k$-th iteration, the algorithm randomly selects a column $\ac{j}$ of $\matA$ for some $j$, and updates $\z^{(k)}$ by projecting it onto the orthogonal complement of the space of $\ac{j}$. The claim is that randomly selecting the columns of $\matA$ with probability proportional to their square norms implies that the algorithm converges to $\bc$ in expectation. After $T$ iterations, the algorithm outputs $\z^{(T)}$ and by orthogonality $\b-\z^{(T)}$ serves as an approximation for $\br$. The next theorem bounds the expected rate of convergence for Algorithm~\ref{alg:randOP}.
\begin{theorem}\label{thm:randOP}
Let $\matA\in\RR^{m\times n}$, $\b\in\RR^m$ and $T>1$ be the input to Algorithm~\ref{alg:randOP}. Fix any integer $k>0$. In exact arithmetic, after $k$ iterations of Algorithm~\ref{alg:randOP} it holds that 
\[\EE \norm{\z^{(k)} - \bc }^2 \leq \left(1 -\frac1{\kappaFS(\matA)}\right)^k  \norm{\br}^2.\]
Moreover, each iteration of Algorithm~\ref{alg:randOP} requires in expectation (over the random choices of the algorithm) at most $5\cavg$ arithmetic operations.
\end{theorem}
\begin{remark}
A suggestion for a stopping criterion for Algorithm~\ref{alg:randOP} is to regularly check: $ \frac{\norm{\matA^\top \z^{(k)}}}{\frobnorm{\matA} \norm{\z^{(k)}}} \leq \eps$ for some given accuracy $\eps>0$. It is easy to see that whenever this criterion is satisfied, it holds that $\norm{\bc-\z^{(k)} } / \norm{\z^{(k)}} \leq \eps \kappaF(\matA)$, i.e., $\b-\z^{(k)}\approx \br$.
\end{remark}
We devote the rest of this subsection to prove Theorem~\ref{thm:randOP}. Define $\matP (j):= \Id_m - \frac{\ac{j} \ac{j}^\top}{\norm{\ac{j}}^2}$ for every $j\in [n]$. Observe that $\matP (j) \matP (j) = \matP (j)$, i.e., $\matP (j)$ is a projector matrix. Let $X$ be a random variable over $\{1,2,\ldots, n\}$ that picks index $j$ with probability $\norm{\ac{j}}^2/\frobnorm{\matA}^2$. It is clear that $\EE [\matP(X)] = \Id_m - \matA\matA^\top /\frobnorm{\matA}^2$. Later we will make use of the following fact.
\begin{fact}\label{lem:technical}
For every vector $\vecu$ in the column space of $\matA$, it holds $\norm{\left(\Id_m - \frac{\matA\matA^\top }{\frobnorm{\matA}^2}\right) \vecu} \leq \left(1 - \frac{\sigma^2_{\min}}{\frobnorm{\matA}^2} \right) \norm{\vecu}$.
\end{fact}
\ignore{
\begin{proof}
	Let $\matA \matA^\top = \sum_i {\sigma_i^2 \vecv_i\vecv_i^\top }$ be the eigenvalue decomposition of $\matA\matA^\top$. Write $\vecu$ in this eigenbasis as $\vecu = \sum_i{\beta_i \vecv_i}$. Then $\norm{\left(\Id_m - \frac{\matA\matA^\top }{\frobnorm{\matA}^2}\right) \vecu}^2 = \norm{\sum_i{(1-\sigma_i^2/ \frobnorm{\matA}^2  ) \beta_i \vecv_i}}^2 \leq \left(1-\sigma^2_{\min}/ \frobnorm{\matA}^2  \right)^2 \norm{\vecu}^2$.
\end{proof}
}
Define $\e^{(k)}:= \z^{(k)} - \bc$ for every $k\geq 0$. A direct calculation implies that 
\[\e^{(k)} = \matP (j_k) \e^{(k-1)}.\]
Indeed, $\e^{(k)} = \z^{(k)} - \bc = \matP(j_k) \z^{(k-1)} - \bc = \matP(j_k) (\e^{(k-1)} + \bc ) - \bc = \matP(j_k) \e^{(k-1)}$ using the definitions of $\e^{(k)}$, $\z^{(k)}$, $\e^{(k-1)}$ and the fact that $\matP (j_k) \bc = \bc$ for any $j_k\in{[n]}$. Moreover, it is easy to see that for every $k\geq0$ $\e^{(k)}$ is in the column space of $\matA$, since $\e^{(0)} = \b - \bc = \br\in \colspan{\matA} $, $\e^{(k)}= \matP (j_k) \e^{(k-1)}$ and in addition $\matP(j_k)$ is a projector matrix for every $j_k\in [n]$.

Let $X_1,X_2,\ldots $ be a sequence of independent and identically distributed random variables distributed as $X$. For ease of notation, we denote by $\EE_{k-1}[\cdot] = \EE_{X_k} [\cdot\ |\ X_1, X_2, \ldots, X_{k-1}]$, i.e., the conditional expectation conditioned on the first $(k-1)$ iteration of the algorithm. It follows that
\begin{align*}
	\EE_{k-1} \norm{ \e^{(k)}}^2 &  =   \EE_{k-1} \norm{ \matP (X_k) \e^{(k-1)} }^2 \ = \ \EE_{k-1} \ip{ \matP (X_k) \e^{(k-1)} }{ \matP (X_k) \e^{(k-1)} } \\
							 &  =   \EE_{k-1} \ip{\e^{(k-1)} }{ \matP (X_k)\matP (X_k) \e^{(k-1)} } \ = \  \ip{\e^{(k-1)} }{ \EE_{k-1} [\matP (X_k)] \e^{(k-1)} } \\
							 &\leq  \norm{\e^{(k-1)}} \norm{ \left(\Id_m - \frac{\matA\matA^\top }{\frobnorm{\matA}^2}\right) \e^{(k-1)} } 
							 \ \leq \ \left(1 - \frac{\sigma^2_{\min}}{\frobnorm{\matA}^2}\right) \norm{\e^{(k-1)}}^2
\end{align*}
where we used linearity of expectation, the fact that $\matP (\cdot)$ is a projector matrix, Cauchy-Schwarz inequality and Fact~\ref{lem:technical}. Repeating the same argument $k-1$ times we get that 
\[\EE\norm{ \e^{(k)}}^2 \leq \left(1 -  \frac1{\kappaFS(\matA)}\right)^k \norm{\e^{(0)}}^2.\]
Note that $\e^{(0)} = \b - \bc = \br$ to conclude. 
%Set $k\geq \kappaFS(\matA)\ln(\norm{\br}^2 /\eps^2)$ and recall the inequality $1-t \leq \exp(-t)$ for $t \leq 1$ to conclude.
%

%
Step $5$ can be rewritten as $\z^{(k+1)} = \z^{(k)} - \left(\ip{\ac{j_k}}{\z^{(k)}} / \norm{\ac{j_k}}^2\right) \ac{j_k}$. At every iteration, the inner product and the update from $\z^{(k)}$ to $\z^{(k+1)}$ require at most $5\nnz{\ac{j_k}}$ operations for some $j_k\in{[n]}$; hence in expectation each iteration requires at most $\sum_{j=1}^{n} 5 p_j \nnz{\ac{j}} = 5\cavg$ operations.
%

%
%%%%%%%%%%%%%%%%%%%%%%%%%%%%%%%%%%%%%%%%%%%%%%%%%%%%%%%
%%%%%%%%%%%%%%%%%%%%%%%%%%%%%%%%%%%%%%%%%%%%%%%%%%%%%%%
\subsection{Randomized Kaczmarz}\label{sec:RK}
%%%%%%%%%%%%%%%%%%%%%%%%%%%%%%%%%%%%%%%%%%%%%%%%%%%%%%%
%%%%%%%%%%%%%%%%%%%%%%%%%%%%%%%%%%%%%%%%%%%%%%%%%%%%%%%
%
%
%%%%%%%%%%%%%%%%%%%%%%%%%%%%%%%%%%%%%%%%%%%%%%%%%%%%%%%
\begin{algorithm}{}
	\caption{Randomized Kaczmarz~\cite{RK}}\label{alg:randomized}
\begin{algorithmic}[1]
\Procedure{}{$\matA$, $\b$, $T$}\Comment{$\matA\in\RR^{m\times n}, \b\in\RR^m$}
\State Set $\x^{(0)}$ to be any vector in the row space of $\matA$
\For {$k=0,1,2,\ldots , T-1$ }
	\State Pick $i_k\in[m]$ with probability $q_i:=\norm{\ar{i}}^2/\frobnorm{\matA}^2, i\in [m]$
	\State Set $ \x^{(k+1)} = \x^{(k)}  + \frac{b_{i_k} - \ip{\x^{(k)}}{\ar{i_k} }}{\norm{\ar{i_k}}^2} \ar{i_k}$
\EndFor
\State Output $\x^{(T)}$
\EndProcedure
\end{algorithmic}
\end{algorithm}
%%%%%%%%%%%%%%%%%%%%%%%%%%%%%%%%%%%%%%%%%%%%%%%%%%%%%%%
%
%
%
%
Strohmer and Vershynin proposed the following randomized variant of Kaczmarz algorithm~(Algorithm~\ref{alg:randomized}), see~\cite{RK} for more details. The following theorem is a restatement of the main result of~\cite{RK} without imposing the full column rank assumption.
%%%%%%%%%%%%%%%%%%%%%%%%%%%%%%%%%%%%%%%%%%%%%%%%%%%%%%%
\begin{theorem}\label{thm:RK:consistent}
Let $\matA\in\RR^{m\times n}$, $\b\in\RR^m$ and $T>1$ be the input to Algorithm~\ref{alg:randomized}. Assume that $\matA \x = \b$ has a solution and denote $\xls:=\pinv{\matA}\b$. In exact arithmetic, Algorithm~\ref{alg:randomized} converges to $\xls$ in expectation:
	\begin{equation}
		\EE \norm{\x^{(k)} - \xls}^2 \leq \left(1 - \frac1{\kappaFS(\matA)}\right)^k \norm{\x^{(0)} - \xls}^2\quad \forall\ k>0.
	\end{equation}
\end{theorem}
%%%%%%%%%%%%%%%%%%%%%%%%%%%%%%%%%%%%%%%%%%%%%%%%%%%%%%%
%%%%%%%%%%%%%%%%%%%%%%%%%%%%%%%%%%%%%%%%%%%%%%%%%%%%%%%
\begin{remark}
The above theorem has been proved in~\cite{RK} for the case of full column rank. Also, the rate of expected convergence in~\cite{RK} is $1-1/\widetilde{\kappa}^2 (\matA)$ where $\widetilde{\kappa}^2(\matA) := \frobnorm{\matA}^2 / \sigma_{\min{(m,n)}}(\matA^\top \matA)$. Notice that if $\rank{\matA}< n$, then $\widetilde{\kappa}^2(\matA)$ is infinite whereas $\kappaFS(\matA)$ is bounded.
\end{remark}
%%%%%%%%%%%%%%%%%%%%%%%%%%%%%%%%%%%%%%%%%%%%%%%%%%%%%%%
%

%
We devote the rest of this subsection to prove Theorem~\ref{thm:RK:consistent} following~\cite{RK}. The proof is based on the following two elementary lemmas which both appeared in~\cite{RK}. However, in our setting, the second lemma is not identical to that in~\cite{RK}. We deferred their proofs to the Appendix.
%%%%%%%%%%%%%%%%%%%%%%%%%%%%%%%%%%%%%%%%%%%%%%%%%%%
\begin{lemma}[Orthogonality]\label{lem:ortho}
Assume that $\matA\x= \b$ has a solution and use the notation of Algorithm~\ref{alg:randomized}, then $\x^{(k+1)} -\xls$ is perpendicular to $\x^{(k+1)} - \x^{(k)}$ for any $k\geq 0$. In particular, in exact arithmetic it holds that $\norm{\x^{(k+1)} - \xls}^2 = \norm{\x^{(k)} - \xls}^2 - \norm{\x^{(k+1)} - \x^{(k)}}^2$.
\end{lemma}
%%%%%%%%%%%%%%%%%%%%%%%%%%%%%%%%%%%%%%%%%%%%%%%%%%%
%

%
The above lemma provides a formula for the error at each iteration. Ideally, we seek to minimize the error at each iteration which is equivalent to maximizing $\norm{\x^{(k+1)} - \x^{(k)}}$ over the choice of the row projections of the algorithm. The next lemma suggests that by randomly picking the rows of $\matA$ reduces the error in expectation.
%%
%%
%%%%%%%%%%%%%%%%%%%%%%%%%%%%%%%%%%%%%%%%%%%%%%%%%%%%%%%
\begin{lemma}[Expected Error Reduction]\label{lem:avg}
Assume that $\matA\x=\b$ has a solution. Let $Z$ be a random variable over $[m]$ with distribution $\Prob{Z=i} = \frac{\norm{\ar{i}}^2}{\frobnorm{\matA}^2}$ and assume that $\x^{(k)}$ is a vector in the row space of $\matA$. If $\x^{(k+1)} := \x^{(k)} + \frac{b_Z - \ip{\x^{(k)}}{\ar{Z}}}{\norm{\ar{Z}}^2} \ar{Z}$ (in exact arithmetic), then
\begin{equation}
\EE_{Z}\norm{\x^{(k+1)} - \xls}^2 \leq \left(1 - \frac1{\kappaFS(\matA)}\right) \norm{\x^{(k)} - \xls}^2.
\end{equation}
\end{lemma}
%%%%%%%%%%%%%%%%%%%%%%%%%%%%%%%%%%%%%%%%%%%%%%%%%%%
%
%%%%%%%%%%%%%%%%%%%%%%%%%%%%%%%%%%%%%%%%%%%%%%%%%%%
%
%
Theorem~\ref{thm:RK:consistent} follows by iterating Lemma \ref{lem:avg}, we get that 
\[\EE \norm{\x^{(k+1)} - \xls}^2 \leq \left(1 - \frac1{\kappaFS(\matA)}\right)^k \norm{\x^{(0)} - \xls}^2.\]
%%%%%%%%%%%%%%%%%%%%%%%%%%%%%%%%%%%%%%%%%%%%%%%%%%%%%%%
%%%%%%%%%%%%%%%%%%%%%%%%%%%%%%%%%%%%%%%%%%%%%%%%%%%%%%%
%
%
%%%%%%%%%%%%%%%%%%%%%%%%%%%%%%%%%%%%%%%%%%%%%%%%%%%%%%%
%%%%%%%%%%%%%%%%%%%%%%%%%%%%%%%%%%%%%%%%%%%%%%%%%%%%%%%
\subsection{Randomized Kaczmarz Applied to Noisy Linear Systems}\label{sec:noisyRK}
%%%%%%%%%%%%%%%%%%%%%%%%%%%%%%%%%%%%%%%%%%%%%%%%%%%%%%%
%%%%%%%%%%%%%%%%%%%%%%%%%%%%%%%%%%%%%%%%%%%%%%%%%%%%%%%
The analysis of Strohmer and Vershynin is based on the restrictive assumption that the linear system has a solution. Needell made a step further and analyzed the more general setting in which the linear system does not have any solution and $\matA$ has full column rank~\cite{Needell09}. In this setting, it turns out that the randomized Kaczmarz algorithm computes an estimate vector that is within a fixed distance from the solution; the distance is proportional to the norm of the ``noise vector'' multiplied by $\kappaFS(\matA)$~\cite{Needell09}. The following theorem is a restatement of the main result in~\cite{Needell09} with two modifications: the full column rank assumption on the input matrix is dropped and the additive term $\gamma$ of Theorem~$2.1$ in~\cite{Needell09} is improved to $\norm{\w}^2/\frobnorm{\matA}^2$. The only technical difference here from~\cite{Needell09} is that the full column rank assumption is not necessary, so we defer the proof to the Appendix for completeness.
\begin{theorem}\label{thm:RK:inconsistent}
Assume that the system $\matA \x = \y$ has a solution for some $\y\in\RR^m$. Denote by $\x^{*} := \pinv{\matA}\y$. Let $\hat\x^{(k)}$ denote the $k$-th iterate of the randomized Kaczmarz algorithm applied to the linear system $\matA \x = \b$ with $\b:=\y + \w$ for any fixed $\w\in\RR^m$, i.e., run Algorithm~\ref{alg:randomized} with input $(\matA,\b)$. In exact arithmetic, it follows that
\begin{equation}\label{ineq:relaxRK}
\EE \norm{\hat\x^{(k)} - \x^{*}}^2 \le \left(1-\frac1{\kappaFS(\matA)}\right)\EE \norm{\hat\x^{(k-1)} - \x^{*}}^2 + \frac{\norm{\w}^2}{\frobnorm{\matA}^2}.
\end{equation}
In particular,
\[\EE \norm{\hat\x^{(k)} - \x^{*}}^2 \le \left(1-\frac1{\kappaFS(\matA)}\right)^k \norm{\x^{(0)}- \x^{*}}^2 + \frac{\norm{\w}^2}{\sigma^2_{\min} }.\]
\end{theorem}
%
%
%
%%%%%%%%%%%%%%%%%%%%%%%%%%%%%%%%%%%%%%%%%%%%%%%%%%%%%%%
%%%%%%%%%%%%%%%%%%%%%%%%%%%%%%%%%%%%%%%%%%%%%%%%%%%%%%%
\section{Randomized Extended Kaczmarz}\label{sec:result}
%%%%%%%%%%%%%%%%%%%%%%%%%%%%%%%%%%%%%%%%%%%%%%%%%%%%%%%
%%%%%%%%%%%%%%%%%%%%%%%%%%%%%%%%%%%%%%%%%%%%%%%%%%%%%%%
%
%
Given any least squares problem, Theorem~\ref{thm:RK:inconsistent} with $\w = \bc$ tells us that the randomized Kaczmarz algorithm works well for least square problems whose least squares error is very close to zero, i.e., $\norm{\w}\approx 0$. Roughly speaking, in this case the randomized Kaczmarz algorithm approaches the minimum $\ell_2$-norm least squares solution up to an additive error that depends on the distance between $\b$ and the column space of $\matA$.

In the present paper, the main observation is that it is possible to efficiently reduce the norm of the ``noisy'' part of $\b$, $\bc$ (using Algorithm~\ref{alg:randOP}) and then apply the randomized Kaczmarz algorithm on a new linear system whose right hand side vector is now arbitrarily close to the column space of $\matA$, i.e., $\matA\x\approx \br$. This idea together with the observation that the least squares solution of the latter linear system is equal (in the limit) to the least squares solution of the original system (see Fact~\ref{fact:xls}) implies a randomized algorithm for solving least squares.

Next we present the randomized extended Kaczmarz algorithm which is a specific combination of the randomized orthogonal projection algorithm together with the randomized Kaczmarz algorithm. 
%%
%
%%%%%%%%%%%%%%%%%%%%%%%%%%%%%%%%%%%%%%%%%%%%%%%%%%%%%%%
%%%%%%%%%%%%%%%%%%%%%%%%%%%%%%%%%%%%%%%%%%%%%%%%%%%%%%%
\subsection{The algorithm}
%%%%%%%%%%%%%%%%%%%%%%%%%%%%%%%%%%%%%%%%%%%%%%%%%%%%%%%
%%%%%%%%%%%%%%%%%%%%%%%%%%%%%%%%%%%%%%%%%%%%%%%%%%%%%%%
%
\begin{algorithm}{}
	\caption{Randomized Extended Kaczmarz (REK)}\label{alg:REK}
\begin{algorithmic}[1]
\Procedure{}{$\matA$, $\b$, $\eps$}\Comment{$\matA\in\RR^{m\times n}, \b\in\RR^m$, $\eps >0$}
\State Initialize $\x^{(0)}=\zeromtx$ and $\z^{(0)} =\b$ 
\For {$k=0,1,2,\ldots $ }
	\State Pick $i_k\in[m]$ with probability $q_i:=\norm{\ar{i}}^2/\frobnorm{\matA}^2, i\in [m]$	
	\State Pick $j_k\in[n]$ with probability $p_j:=\norm{\ac{j}}^2/\frobnorm{\matA}^2,\ j\in [n]$
	\State Set $ \z^{(k+1)} = \z^{(k)} - \frac{\ip{\ac{j_k}}{\z^{(k)}}}{\norm{\ac{j_k}}^2}\ac{j_k}  $
	\State Set $ \x^{(k+1)} = \x^{(k)}  + \frac{b_{i_k} - z^{(k)}_{i_k} - \ip{\x^{(k)}}{\ar{i_k} }}{\norm{\ar{i_k}}^2} \ar{i_k}$
	\State\label{alg:stopping} Check every $8 \min (m,n)$ iterations and terminate if it holds:
	\[  \frac{\norm{\matA \x^{(k)} - (\b - \z^{(k)}) }}{\frobnorm{\matA} \norm{\x^{(k)}}}  \leq \eps \quad\text{and} \quad \frac{\norm{\matA^\top \z^{(k)}}}{\frobnorm{\matA}^2\norm{\x^{(k)}}} \leq \eps.\]
	
\EndFor
\State Output $\x^{(k)}$
\EndProcedure
\end{algorithmic}
\end{algorithm}
%%%%%%%%%%%%%%%%%%%%%%%%%%%%%%%%%%%%%%%%%%%%%%%%%%%%%%%
%%%%%%%%%%%%%%%%%%%%%%%%%%%%%%%%%%%%%%%%%%%%%%%%%%%%%%%
%
%
We describe a randomized algorithm that converges in expectation to the minimum $\ell_2$-norm solution vector $\xls$ (Algorithm~\ref{alg:REK}). The proposed algorithm consists of two components. The first component consisting of Steps $5$ and $6$ is responsible to implicitly maintain an approximation to $\br$ formed by $\b-\z^{(k)}$. The second component, consisting of Steps 4 and 7, applies the randomized Kaczmarz algorithm with input $\matA$ and the current approximation $\b -\z^{(k)}$ of $\br$, i.e., applies the randomized Kaczmarz on the system $\matA \x = \b - \z^{(k)}$. Since $\b - \z^{(k)}$ converges to $\br$, $\x^{(k)}$ will eventually converge to the minimum Euclidean norm solution of $\matA \x = \br$ which equals to $\xls=\pinv{\matA}\b$ (see Fact~\ref{fact:xls}).

The stopping criterion of Step~\ref{alg:stopping} was decided based on the following analysis. Assume that the termination criteria are met for some $k>0$. Let $\z^{(k)} = \bc + \w$ for some $\w\in\colspan{\matA}$ (which holds by the definition of $\z^{(k)}$). Then,
\begin{align*}
	\norm{\matA^\top \z^{(k)}} &= \norm{\matA^\top (\bc + \w)} = \norm{\matA^\top \w} \geq \sigma_{\min}(\matA)\norm{\z^{(k)} - \bc}.
\end{align*}
By re-arranging terms and using the second part of the termination criterion, it follows that $\norm{\z^{(k)} - \bc} \leq \eps \frac{\frobnorm{\matA}^2}{\sigma_{\min}} \norm{\x^{(k)}}$. Now,
\begin{align*}
	\norm{\matA(\x^{(k)} - \xls)} & \leq \norm{\matA\x^{(k)} - (\b - \z^{(k)}) } + \norm{ \b- \z^{(k)} - \matA\xls} \\
								  & \leq \eps \frobnorm{\matA} \norm{\x^{(k)}} + \norm{\bc - \z^{(k)}} \\
								  & \leq \eps \frobnorm{\matA} \norm{\x^{(k)}} + \eps \frac{\frobnorm{\matA}^2}{\sigma_{\min}} \norm{\x^{(k)}},
\end{align*}
where we used the triangle inequality, the first part of the termination rule together with $\br = \matA\xls$ and the above discussion.
Now, since $\x^{(k)},\xls \in\colspan{\matA^\top}$, it follows that 
\begin{equation}\label{REK:forwardErr}
			\frac{\norm{\x^{(k)} - \xls } }{\norm{\x^{(k)}}} \leq \eps \kappaF(\matA) ( 1 + \kappaF(\matA)).
\end{equation}
Equation~\eqref{REK:forwardErr} demonstrates that the forward error of REK after termination is bounded.
%
%%%%%%%%%%%%%%%%%%%%%%%%%%%%%%%%%%%%%%%%%%%%%%%%%%%%%%%
%%%%%%%%%%%%%%%%%%%%%%%%%%%%%%%%%%%%%%%%%%%%%%%%%%%%%%%
\subsection{Rate of convergence}
%%%%%%%%%%%%%%%%%%%%%%%%%%%%%%%%%%%%%%%%%%%%%%%%%%%%%%%
%%%%%%%%%%%%%%%%%%%%%%%%%%%%%%%%%%%%%%%%%%%%%%%%%%%%%%%
%
%
The following theorem bounds the expected rate of convergence of Algorithm~\ref{alg:REK}.
\begin{theorem}\label{thm:REK}
After $T>1$ iterations, in exact arithmetic, Algorithm~\ref{alg:REK} with input $\matA$ (possibly rank-deficient) and $\b$ computes a vector $\x^{(T)}$ such that
\[ \EE \norm{\x^{(T)} - \xls}^2 \leq \left(1 - \frac1{\kappaFS(\matA)}\right)^{\lfloor T/2\rfloor}\left(1 + 2\cond{\matA}\right) \norm{\xls}^2.\]
\end{theorem}
%
%
%%%%%%%%%%%%%%%%%%%%%%%%%%%%%%%%%%%%%%%%%%%%%%%%%%%%%%%
\begin{proof}
%%%%%%%%%%%%%%%%%%%%%%%%%%%%%%%%%%%%%%%%%%%%%%%%%%%%%%%
%
%
For the sake of notation, set $\alpha =1- 1/\kappaFS(\matA)$ and denote by $\EE_{k}[\cdot] := \EE[\cdot \ |\ i_0,j_0, i_1,j_1,\ldots , i_k, j_k]$, i.e., the conditional expectation with respect to the first $k$ iterations of Algorithm~\ref{alg:REK}. Observe that Steps $5$ and $6$ are independent from Steps $4$ and $7$ of Algorithm~\ref{alg:REK}, so Theorem~~\ref{thm:randOP} implies that for every $l\geq 0$
\begin{equation}\label{eq:improve}
\EE \norm{\z^{(l)} - \bc }^2 \leq \alpha^l \norm{\br}^2 \leq \norm{\br}^2.
\end{equation}
Fix a parameter $k^* := \lfloor T/2 \rfloor$. After the $k^*$-th iteration of Algorithm~\ref{alg:REK}, it follows from Theorem~\ref{thm:RK:inconsistent} (Inequality~\eqref{ineq:relaxRK}) that
\begin{align*}
	\EE_{(k^*-1)} \norm{\x^{(k^*)} - \xls}^2 & \leq  \alpha  \norm{\x^{(k^* - 1)} - \xls}^2 + \frac{ \norm{\bc - \z^{(k^* - 1)}}^2}{\frobnorm{\matA}^2}.
\end{align*}
Indeed, the randomized Kaczmarz algorithm is executed with input $(\matA, \b - \z^{(k^*-1)})$ and current estimate vector $\x^{(k^* -1)}$. Set $\y=\br$ and $\w=\bc - \z^{(k^*-1)}$ in Theorem~\ref{thm:RK:inconsistent} and recall that $\xls=\pinv{\matA} \b = \pinv{\matA}\br = \pinv{\matA}\y$.
Now, averaging the above inequality over the random variables $i_1,j_1, i_2, j_2, \ldots, i_{k^* - 1}, j_{k^* - 1}$ and using linearity of expectation, it holds that
\begin{align}
	\EE \norm{\x^{(k^*)} - \xls}^2 & \leq  \alpha  \EE\norm{\x^{(k^* - 1)} - \xls}^2 + \frac{ \EE\norm{\bc - \z^{(k^* - 1)}}^2}{\frobnorm{\matA}^2}\label{eq:star}\\
	& \leq  \alpha \EE \norm{\x^{(k^* - 1)} - \xls}^2 + \frac{\norm{\br}^2}{\frobnorm{\matA}^2} \quad  \text{by Ineq.}~\eqref{eq:improve}\nonumber\\
	& \leq  \ldots \leq \alpha^{k^*}\norm{\x^{(0)} - \xls}^2 + \sum_{l=0}^{k^* - 2} \alpha^l \frac{\norm{\br}^2}{\frobnorm{\matA}^2},\quad (\text{repeat the above } k^* - 1 \text{ times}) \nonumber \\
	& \leq  \norm{\xls}^2 + \sum_{l=0}^{\infty} \alpha^l \frac{\norm{\br}^2}{\frobnorm{\matA}^2}, \quad \text{since } \alpha <1\text{ and }\x^{(0)} = \zero . \nonumber
\end{align}
Simplifying the right hand side using the fact that $\sum_{l=0}^{\infty} \alpha^l = \frac1{1-\alpha} = \kappaFS(\matA)$, it follows
\begin{equation}\label{eq:stopTime}
		\EE \norm{\x^{(k^*)} - \xls}^2 \leq  \norm{\xls}^2 + \norm{\br}^2/\sigma_{\min}^2.
\end{equation}
Moreover, observe that for every $l\geq 0$
\begin{equation}\label{eq:better}
\EE \norm{\bc - \z^{(l+k^*)}}^2 \leq \alpha^{l+k^*} \norm{\br}^2 \leq \alpha^{k^*}\norm{\br}^2.
\end{equation}
Now for any $k>0$, similar considerations as Ineq.~\eqref{eq:star} implies that
\begin{align*}
\EE \norm{\x^{(k + k^*)} - \xls}^2  & \leq   \alpha \EE \norm{\x^{(k + k^* - 1)} - \xls}^2 + \frac{\EE \norm{\bc - \z^{(k - 1 + k^*)}}^2}{\frobnorm{\matA}^2} \\
									& \leq   \ldots \leq \alpha^k \EE \norm{\x^{(k^*)} - \xls}^2 + \sum_{l=0}^{k - 1}\alpha^{(k - 1) - l} \frac{\EE \norm{\bc - \z^{(l + k^*)}}^2}{\frobnorm{\matA}^2} \quad \text{(by induction)}\\
									& \leq   \alpha^k \EE \norm{\x^{(k^*)} - \xls}^2 + \frac{\alpha^{k^*}\norm{\br}^2}{\frobnorm{\matA}^2} \sum_{l=0}^{k - 1}\alpha^{l} \quad  \text{(by Ineq.~\eqref{eq:better})} \\
									& \leq   \alpha^k \left(\norm{\xls}^2 + \norm{\br}^2/\sigma_{\min}^2\right)  + \alpha^{k^*}\norm{\br}^2 / \sigma_{\min}^2 \quad  \left(\text{by Ineq.~\eqref{eq:stopTime}}\right) \\
									&   =   \alpha^k \norm{\xls}^2 + (\alpha^{k}+\alpha^{k^*}) \norm{\br}^2/\sigma_{\min}^2\\
									&   \leq   \alpha^k \norm{\xls}^2 + (\alpha^{k}+\alpha^{k^*}) \cond{\matA} \norm{\xls}^2 \quad \text{since }\norm{\br} \leq \sigma_{\max} \norm{\xls} \\
									&   \leq \alpha^{k^*} (1 + 2\cond{\matA})\norm{\xls}^2.
\end{align*}
To derive the last inequality, consider two cases. If $T$ is even, set $k=k^*$, otherwise set $k=k^*+1$. In both cases, $(\alpha^{k}+\alpha^{k^*}) \leq 2\alpha^{k^*}$.
%
%%%%%%%%%%%%%%%%%%%%%%%%%%%%%%%%%%%%%%%%%%%%%%%%%%%%%%%
\end{proof}
%%%%%%%%%%%%%%%%%%%%%%%%%%%%%%%%%%%%%%%%%%%%%%%%%%%%%%%

%
%%%%%%%%%%%%%%%%%%%%%%%%%%%%%%%%%%%%%%%%%%%%%%%%%%%%%%%
%%%%%%%%%%%%%%%%%%%%%%%%%%%%%%%%%%%%%%%%%%%%%%%%%%%%%%%
\subsection{Theoretical bounds on time complexity}
%%%%%%%%%%%%%%%%%%%%%%%%%%%%%%%%%%%%%%%%%%%%%%%%%%%%%%%
%%%%%%%%%%%%%%%%%%%%%%%%%%%%%%%%%%%%%%%%%%%%%%%%%%%%%%%
%
In this section, we discuss the running time complexity of the randomized extended Kaczmarz (Algorithm~\ref{alg:REK}). Recall that REK is a Las-Vegas randomized algorithm, i.e., the algorithm always outputs an ``approximately correct'' least squares estimate (satisfying~\eqref{REK:forwardErr}) but its runnning time is a random variable. Given any fixed accuracy parameter $\eps>0$ and any fixed failure probability $0<\delta<1$ we bound the number of iterations required by the algorithm to terminate with probability at least $1-\delta$.
\begin{lemma}\label{lem:runtime}
Fix an accuracy parameter $0<\eps<2$ and failure probability $0<\delta<1$. In exact arithmetic, Algorithm~\ref{alg:REK} terminates after at most
	\[T^*:=2\kappaFS(\matA) \ln \left(\frac{32(1+2\cond{\matA})}{\delta\eps^2}\right)\]
iterations with probability at least $1-\delta$.
\end{lemma}
\begin{proof}
	Denote $\alpha := 1 - 1/ \kappaFS(\matA)$ for notational convenience. It suffices to prove that with probability at least $1-\delta$ the conditions of Step~\ref{alg:stopping} of Algorithm~\ref{alg:REK} are met. Instead of proving this, we will show that:
\begin{enumerate}
	\item With probability at least $1-\delta/2$: $\norm{(\b - \z^{(T^*)}) - \br} \leq \eps\norm{\br} / 4$.
	\item With probability at least $1-\delta/2$: $\norm{\x^{(T^*)}-\xls} \leq \eps \norm{\xls}/4$.
\end{enumerate}
Later we prove that Items (1) and (2) imply the Lemma. First we prove Item (1). By the definition of the algorithm, 
\begin{align*}
	\Prob{ \norm{ (\b - \z^{(T^*)}) - \br } \geq \eps\norm{\br} / 4} & =  \Prob{ \norm{\bc - \z^{(T^*)}  }^2 \geq \eps^2\norm{\br}^2/16} \\
																	& \leq  \frac{16 \EE \norm{  \z^{(T^*)} - \bc }^2}{\eps^2\norm{\br}^2} \\
																	& \leq  16\alpha^{T^*}/\eps^2  \leq \delta/2
\end{align*}
the first equality follows since $\b -\br = \bc$, the second inequality is Markov's inequality, the third inequality follows by Theorem~\ref{thm:randOP}, and the last inequality since $T^*\geq \kappaFS(\matA) \ln(\frac{32}{\delta \eps^2})$.

Now, we prove Item (2):
\begin{align*}
		\Prob{ \norm{\x^{(T^*)}-\xls} \leq \eps\norm{\xls}/4} & \leq  \frac{ 16\EE \norm{\x^{(T^*)}-\xls}^2 }{\eps^2\norm{\xls}^2} \\
																					& \leq  16\alpha^{\lfloor T^*/2\rfloor} (1+2\cond{\matA})/ \eps^2  \leq \delta/2.
\end{align*}
the first inequality is Markov's inequality, the second inequality follows by Theorem~\ref{thm:REK}, and the last inequality follows provided that $T^* \geq 2\kappaFS(\matA) \ln \left(\frac{32(1+2\cond{\matA})}{\delta\eps^2}\right)$
A union bound on the complement of the above two events (Item (1) and (2)) implies that both events happen with probability at least $1-\delta$. Now we show that conditioning on Items (1) and (2), it follows that REK terminates after $T^*$ iterations, i.e.,
\[ \norm{\matA \x^{(T^*)} - (\b - \z^{(T^*)}) } \leq \eps \frobnorm{\matA}\norm{\x^{(T^*)}}\quad \text{and}\quad \frac{\norm{\matA^\top \z^{(k)}}}{\frobnorm{\matA}^2\norm{\x^{(k)}}} \leq \eps.\]
We start with the first condition. First, using triangle inequality and Item 2, it follows that
\begin{equation}\label{ineq:xlsxk}
\norm{\x^{(T^*)}} \geq  \norm{\xls} - \norm{\xls-\x^{(T^*)}}  \geq (1 - \eps/4 ) \norm{\xls}.
\end{equation}
Now,
\begin{align*}
	\norm{\matA \x^{(T^*)} - (\b - \z^{(T^*)}) } & \leq \norm{\matA \x^{(T^*)} - \br} + \norm{ (\b - \z^{(T^*)}) - \br } \\
	 											 & \leq \norm{\matA (\x^{(T^*)} - \xls)} + \eps \norm{\br} / 4 \\
		 										 & \leq \sigma_{\max} \norm{\x^{(T^*)} - \xls} + \eps \norm{\matA\xls} / 4 \\
												 & \leq \eps \sigma_{\max} \norm{\xls} /2 \\
												 & \leq \frac{\eps/2}{1- \eps/4} \norm{\x^{(T^*)}} \leq \eps \norm{\x^{(T^*)}}
\end{align*}
where the first inequality is triangle inequality, the second inequality follows by Item $1$ and $\br =\matA\xls$, the third and forth inequality follows by Item $2$ and the fifth inequality holds by Inequality~\eqref{ineq:xlsxk} and the last inequality follows since $\eps < 2$.
The second condition follows since
\begin{align*}
	\norm{\matA^\top \z^{(T^*)}} & = \norm{\matA^\top (\bc - \z^{(T^*)})}  \leq \sigma_{\max} \norm{\bc - \z^{(T^*)}} \\
								 & \leq \eps \sigma_{\max} \norm{\br}/4  \leq \eps \sigma^2_{\max} \norm{\xls}/4 \\
								 & \leq \frac{\eps/4}{1-\eps /4} \sigma^2_{\max} \norm{\x^{(T^*)} } \leq \eps \frobnorm{\matA}^2 \norm{\x^{(T^*)} }.
\end{align*}
the first equation follows by orthogonality, the second inequality assuming Item (2), the third inequality follows since $\br=\matA\xls$, the forth inequality follows by~\eqref{ineq:xlsxk} and the final inequality since $\eps <2$.
\end{proof}

Lemma~\ref{lem:runtime} bounds the number of iterations with probability at least $1-\delta$, next we bound the total number of arithmetic operations in worst case~(Eqn.~\eqref{eq:runtime:det}) and in expectation~(Eqn.~\eqref{eq:runtime:rand}). Let's calculate the computational cost of REK in terms of floating-point operations (flops) per iteration. For the sake of simplicity, we ignore the additional (negligible) computational overhead required to perform the sampling operations (see Section~\ref{sec:impl} for more details) and checking for convergence. 

Each iteration of Algorithm~\ref{alg:REK} requires four level-1 BLAS operations (two \emph{DDOT} operations of size $m$ and $n$, respectively, and two \emph{DAXPY} operations of size $n$ and $m$, respectively) and additional four flops. In total, $4(m+n)+2$ flops per iteration. 
Therefore by Lemma~\ref{lem:runtime}, with probability at least $1-\delta$, REK requires at most 
\begin{equation}\label{eq:runtime:det}
5 (m +n ) \cdot T^* \leq  10 (m+n) \rank{\matA}  \cond{\matA} \ln \left(\frac{32(1+2\cond{\matA})}{\delta\eps^2}\right)
\end{equation}
arithmetic operations (using that $\kappaFS(\matA) \leq \rank{\matA} \cond{\matA}$).
Next, we bound the \emph{expected running time} of REK for achieving the above guarantees for any fixed $\eps$ and $\delta$. Obviously, the expected running time is at most the quantity in~\eqref{eq:runtime:det}. However, as we will see shortly the expected running time is proportional to $\nnz{\matA}$ instead of $(m+n)\rank{\matA}$.

Exploiting the (possible) sparsity of $\matA$, we first show that each iteration of Algorithm~\ref{alg:REK} requires at most $5(\cavg + \ravg)$ operations in expectation. For simplicity of presentation, we assume that we have stored $\matA$ in compressed column sparse format \emph{and} compressed row sparse format~\cite{book:templates}.

Indeed, fix any $i_k\in{[m]}$ and $j_k\in{[n]}$ at some iteration $k$ of Algorithm~\ref{alg:REK}. Since $\matA$ is both stored in compressed column and compressed sparse format, Steps $7$ and Step $8$ can be implemented in $5 \nnz{\ac{j_k}}$ and 5$\nnz{\ar{i_k}}$, respectively.

By the linearity of expectation and the definitions of $\cavg$ and $\ravg$, the expected running time after $T^*$ iterations is at most $5T^* (\cavg + \ravg)$. It holds that (recall that $p_j = \norm{\ac{j}}^2/\frobnorm{\matA}^2$)
\begin{align*}
	\cavg T^* & =  \frac{2}{\frobnorm{\matA}^2}\left(\sum_{j=1}^{n} \norm{\ac{j}}^2 \nnz{\ac{j}}\right)\frac{\frobnorm{\matA}^2}{\sigma_{\min}^2}\ln \left(\frac{32(1+2\cond{\matA})}{\delta\eps^2}\right)\\
			  & =  2\frac{\sum_{j=1}^{n} \norm{\ac{j}}^2 \nnz{\ac{j}}}{\sigma_{\min}^2}\ln \left(\frac{32(1+2\cond{\matA})}{\delta\eps^2}\right)\\
			  & \leq 2\sum_{j=1}^{n}\nnz{\ac{j}}\frac{\max_{j\in{[n]}} \norm{\ac{j}}^2}{\sigma_{\min}^2} \ln \left(\frac{32(1+2\cond{\matA})}{\delta\eps^2}\right)\\
			& \leq 2\nnz{\matA}\cond{\matA} \ln \left(\frac{32(1+2\cond{\matA})}{\delta\eps^2}\right)
\end{align*}
using the definition of $\cavg$ and $T^*$ in the first equality and the fact that $\max_{j\in{[n]}} \norm{\ac{j}}^2 \leq \sigma_{\max}^2$ and $\sum_{j=1}^{n}\nnz{\ac{j}} = \nnz{\matA}$ in the first and second inequality. A similar argument shows that $\ravg T^* \leq 2\nnz{\matA} \cond{\matA} \ln \left(\frac{32(1+2\cond{\matA})}{\delta\eps^2}\right)$ using the inequality $\max_{i\in{[m]}} \norm{\ar{i}}^2 \leq \sigma_{\max}^2$.
Hence by Lemma~\ref{lem:runtime}, with probability at least $1-\delta$, the expected number of arithmetic operations of REK is at most 
\begin{equation}\label{eq:runtime:rand}
 20 \nnz{\matA} \cond{\matA}\ln \left(\frac{32(1+2\cond{\matA})}{\delta\eps^2}\right).
\end{equation}
In other words, the expected running time analysis is much tighter than the worst case displayed in Equation~\eqref{eq:runtime:det} and is proportional to $\nnz{\matA}$ times the square condition number of $\matA$ as advertised in the abstract.
%
%%%%%%%%%%%%%%%%%%%%%%%%%%%%%%%%%%%%%%%%%%%%%%%%%%%%%%%
%%%%%%%%%%%%%%%%%%%%%%%%%%%%%%%%%%%%%%%%%%%%%%%%%%%%%%%
\section{Implementation and Experimental Results}\label{sec:impl}
%%%%%%%%%%%%%%%%%%%%%%%%%%%%%%%%%%%%%%%%%%%%%%%%%%%%%%%
%%%%%%%%%%%%%%%%%%%%%%%%%%%%%%%%%%%%%%%%%%%%%%%%%%%%%%%
%

%
%%%%%%%%%%%%%%%%%%%%%%%%%%%%%%%%%%%%%%%%%%%%%%%%%%%%%%%
%%%%%%%%%%%%%%%%%%%%%%%%%%%%%%%%%%%%%%%%%%%%%%%%%%%%%%%
\subsection{Implementation}
%%%%%%%%%%%%%%%%%%%%%%%%%%%%%%%%%%%%%%%%%%%%%%%%%%%%%%%
%%%%%%%%%%%%%%%%%%%%%%%%%%%%%%%%%%%%%%%%%%%%%%%%%%%%%%%
%
The proposed algorithm has been entirely implemented in C. We provide three implementation of Algorithm~\ref{alg:REK}: REK-C, REK-BLAS and REK-BLAS-PRECOND. REK-C corresponds to a direct translation of Algorithm~\ref{alg:REK} to C code. REK-BLAS is an implementation of REK with two additional technical features. First, REK-BLAS uses level-1 BLAS routines for all operations of Algorithm~\ref{alg:REK} and secondly REK-BLAS additionally stores explicitly the transpose of $\matA$ for more efficiently memory access of both the rows and columns of $\matA$ using BLAS. REK-BLAS-PRECOND is an implementation of REK-BLAS that additionally supports upper triangular preconditioning; we used Blendenpik's preconditioning code to ensure a fair comparison with Blendenpik (see Section~\ref{sec:exp}). In the implementations of REK-C and REK-BLAS we check for convergence every $8\min (m,n)$ iterations. 

Moreover, all implementations include efficient code that handles sparse input matrices using the compressed column (and row) sparse matrix format~\cite{book:templates}.

\paragraph{Sampling from non-uniform distributions}
The sampling operations of Algorithm~\ref{alg:REK} (Steps 4 and 5) are implemented using the so-called ``alias method'' for generating samples from any given discrete distribution~\cite{sampling:Walker,random:Alias}. The alias method, assuming access to a uniform random variable on $[0,1]$ in constant time and linear time preprocessing, generates one sample of the given distribution in constant time~\cite{random:Alias}. We use an implementation of W. D. Smith that is described in~\cite{aliasMethod:C} and C's \emph{drand48}() to get uniform samples from $[0,1]$.
\begin{figure}[h]
%\centering
     \subfigure[Random sparse matrices having 800 columns and density 0.25.]{\scalebox{1.0}{\label{fig:sparseOver}\includegraphics[width=0.5\textwidth]{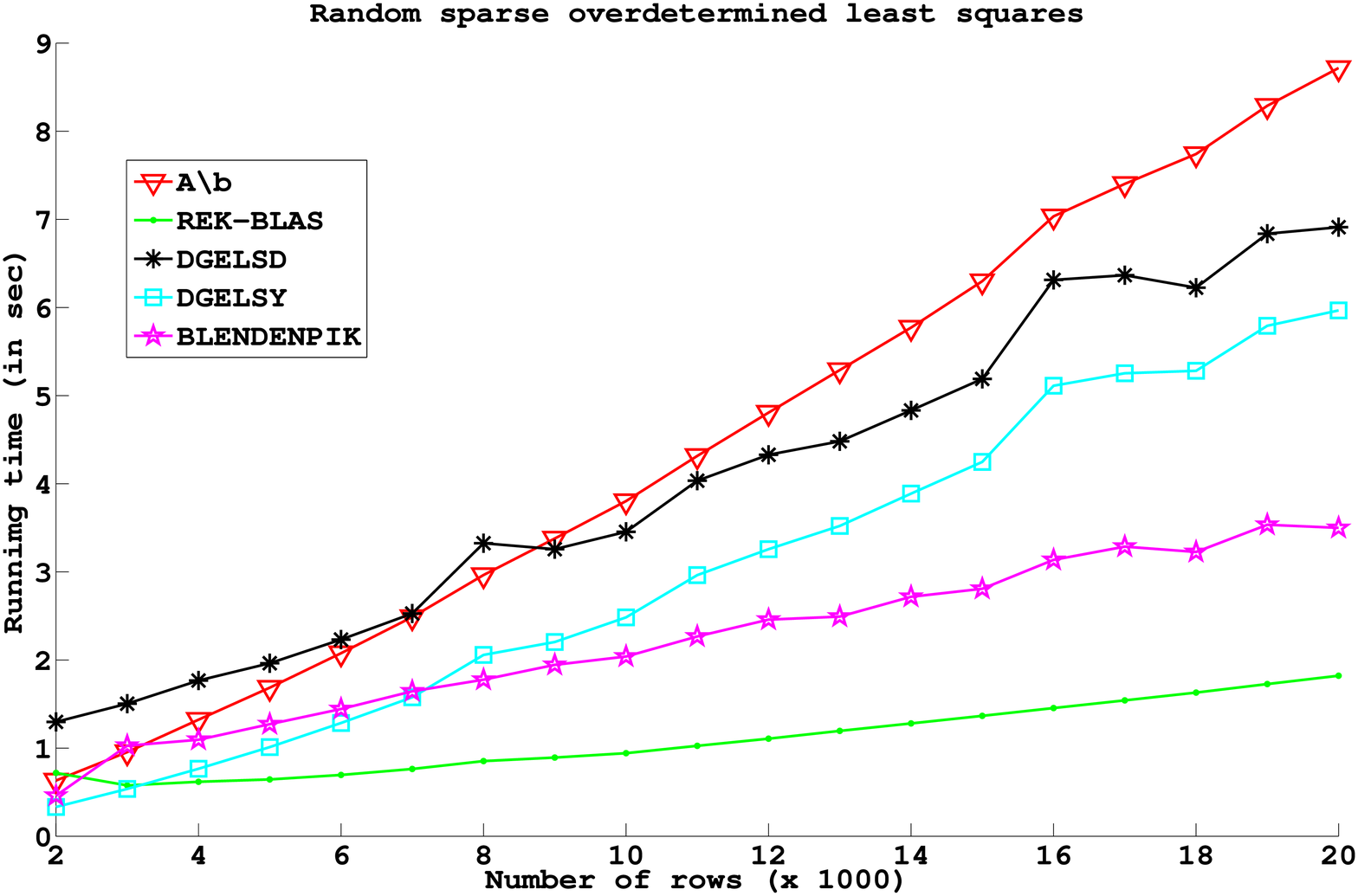}}}
     \subfigure[Random sparse matrices having 800 rows and density 0.25.]
{\scalebox{1.0}{\label{fig:sparseUnder}\includegraphics[width=0.5\textwidth]{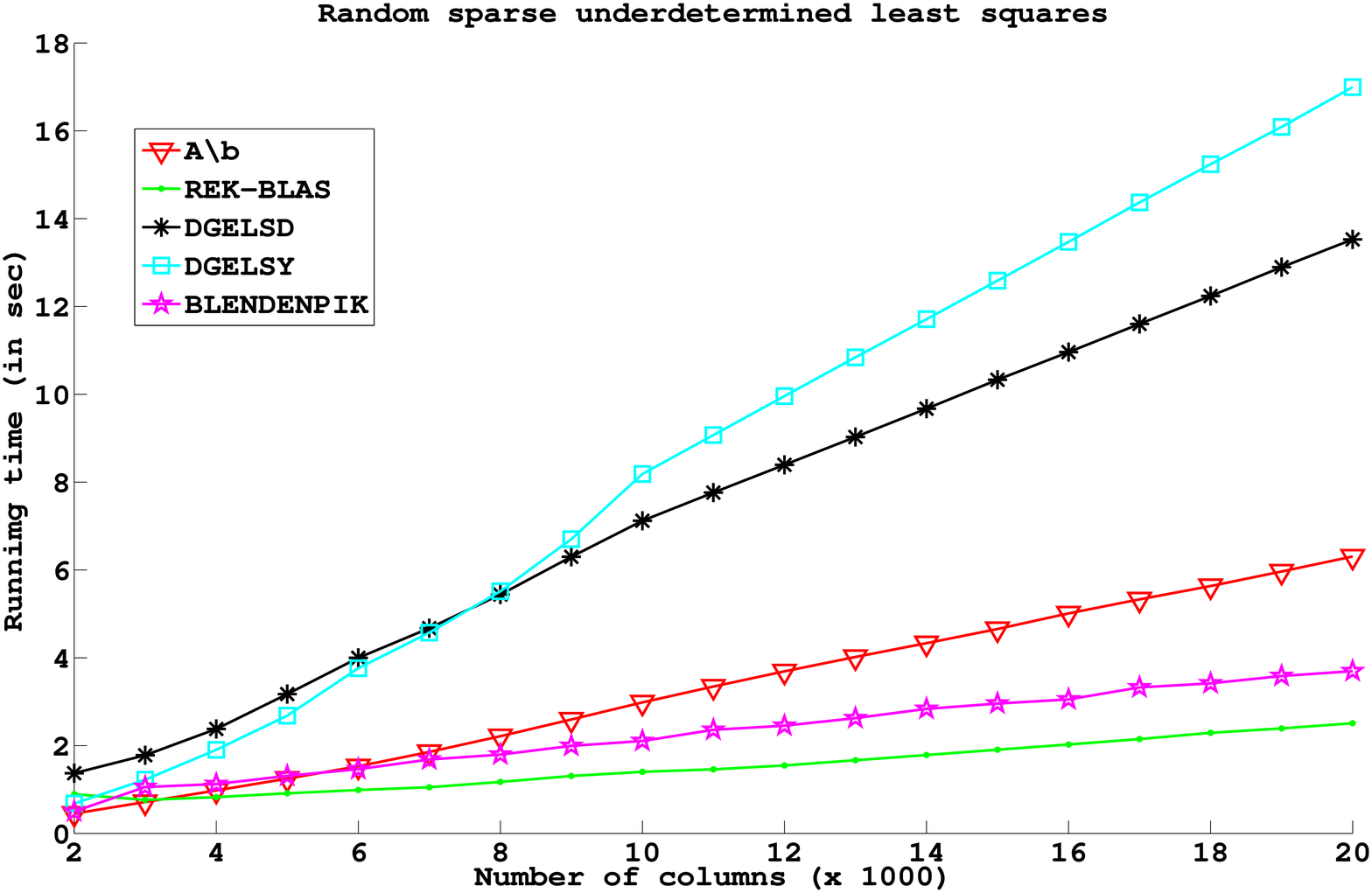}}} 
       %add desired spacing between images, e. g. ~, \quad, \qquad etc.
\caption{Figures depict the running time (in seconds) vs increasing number of rows/columns (scaled by 1000) for the case of random sparse overdetermined (Figure~\ref{fig:sparseOver}) and underdetermined (Figure~\ref{fig:sparseUnder}) least squares problems.}\label{fig:sparse}
\end{figure}

\begin{figure}[h]
%\centering
     \subfigure[Random dense matrices having 500 columns.]{\scalebox{1.0}{\label{fig:denseOver}\includegraphics[width=0.5\textwidth]{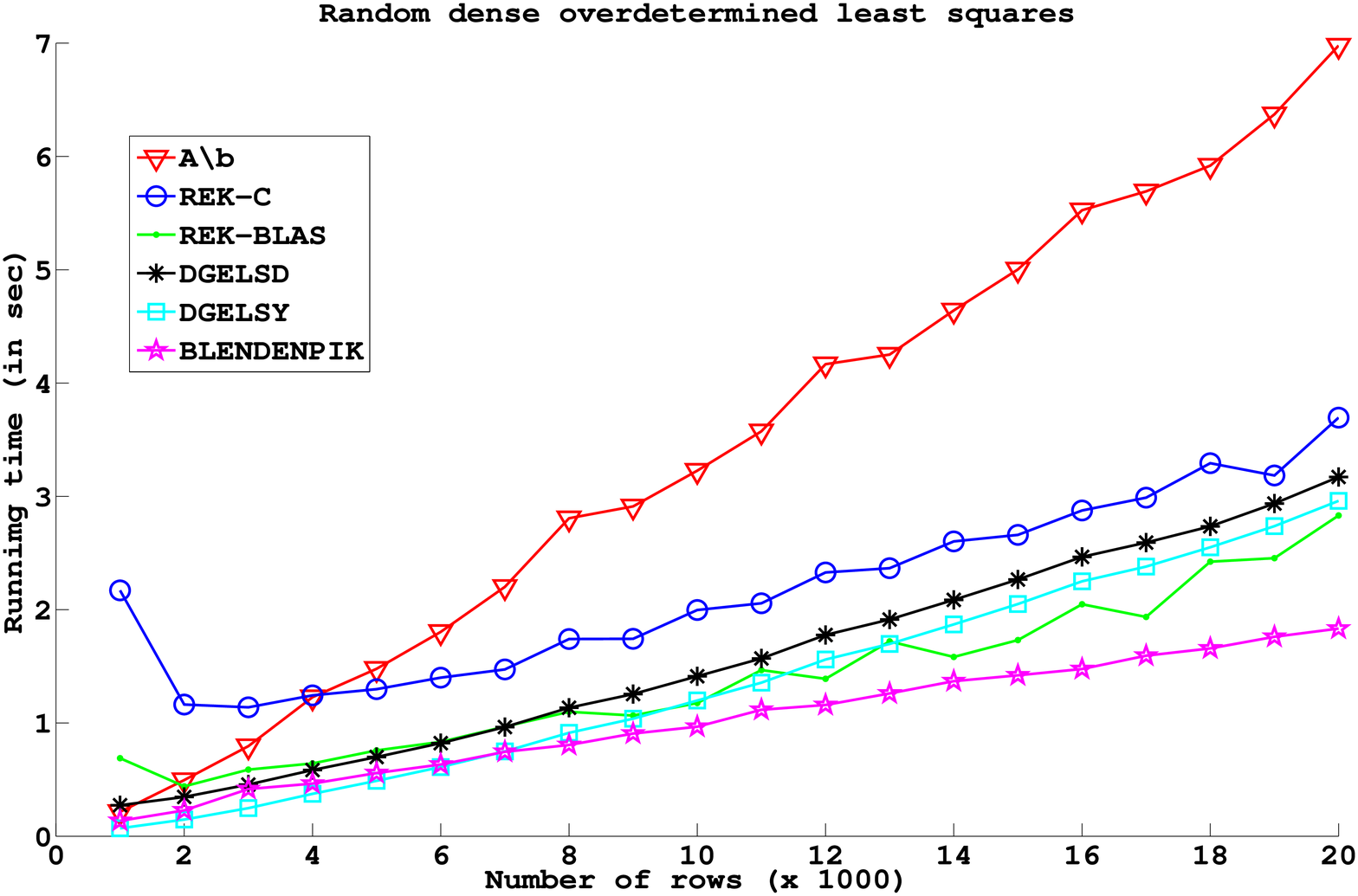}}}
     \subfigure[Random dense matrices having 500 rows.]{\scalebox{1.0}{\label{fig:denseUnder}\includegraphics[width=0.5\textwidth]{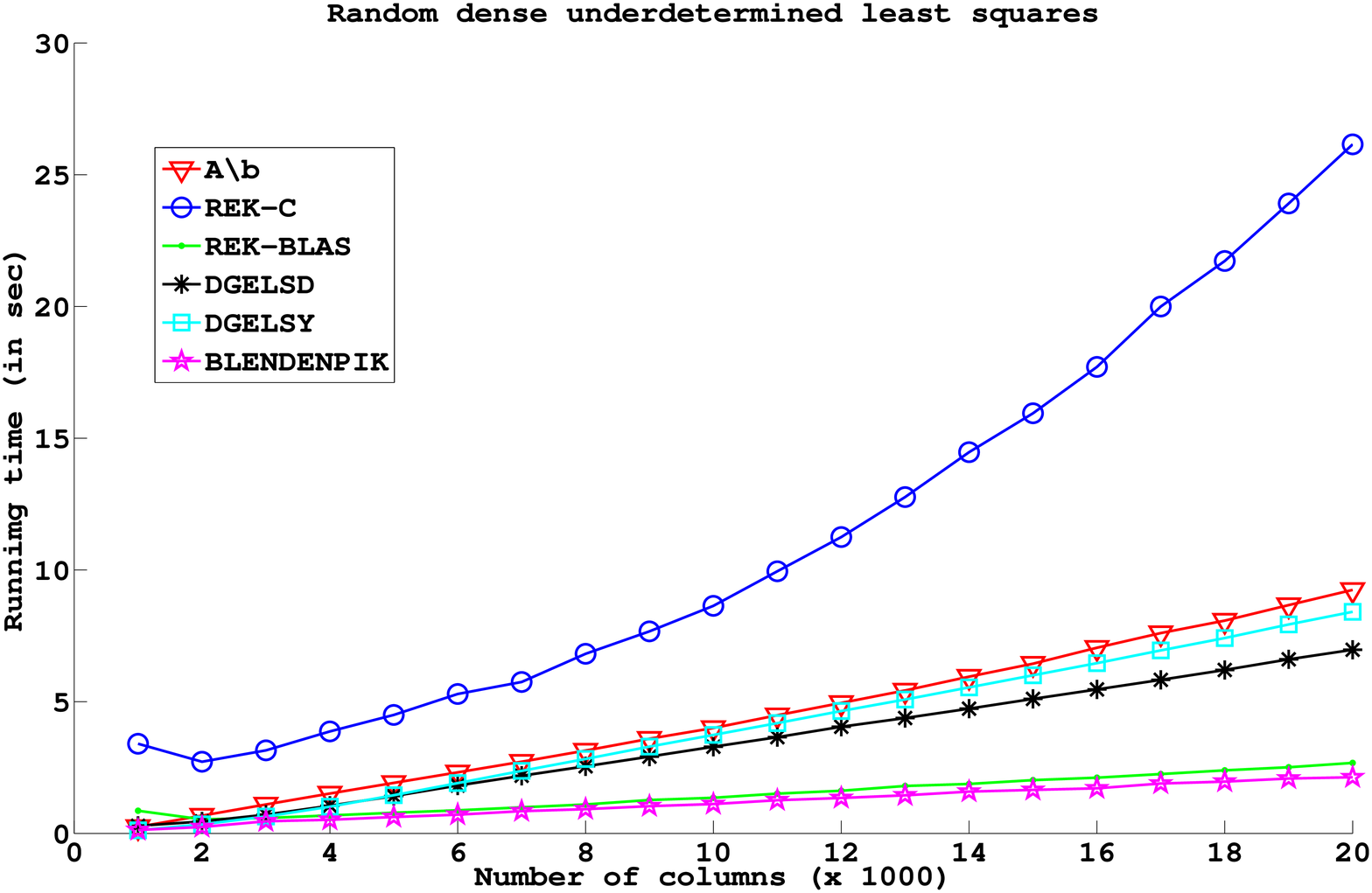}}} 
       %add desired spacing between images, e. g. ~, \quad, \qquad etc.
\caption{Figures depict the running time (in seconds) vs increasing number of rows/columns (scaled by 1000) for the case of random dense overdetermined (Figure~\ref{fig:denseOver}) and underdetermined (Figure~\ref{fig:denseUnder}) least squares problems.}\label{fig:dense}
\end{figure}

%%%%%%%%%%%%%%%%%%%%%%%%%%%%%%%%%%%%%%%%%%%%%%%%%%%%%%%
%%%%%%%%%%%%%%%%%%%%%%%%%%%%%%%%%%%%%%%%%%%%%%%%%%%%%%%
\subsection{Experimental Results}\label{sec:exp}
%%%%%%%%%%%%%%%%%%%%%%%%%%%%%%%%%%%%%%%%%%%%%%%%%%%%%%%
%%%%%%%%%%%%%%%%%%%%%%%%%%%%%%%%%%%%%%%%%%%%%%%%%%%%%%%
We report our experimental results in this section. We compared the randomized extended Kaczmarz (REK-C, REK-BLAS,REK-BLAS-PRECOND) algorithm to LAPACK's DGELSY and DGELSD least squares solvers, Blendenpick\footnote{Available at http://www.mathworks.com/matlabcentral/fileexchange/25241-blendenpik. Blendenpik's default settings were used.} (version 1.3,~\cite{AMT10}) and MATLAB's backslash operator. LSRN~\cite{lsrn} did not perform well under a setup in which no parallelization is allowed as the one used here, so we do not include LSRN's performance. DGELSY uses QR factorization with pivoting and DGELSD uses the singular value decomposition. We use MATLAB with version 7.9.0.529 (R2009b). In addition, we use MATLAB's included BLAS and LAPACK packages and we call LAPACK's functions from MATLAB using MATLAB's CMEX technology which allows us to measure only LAPACK's elapsed time. We should highlight that MATLAB is used as a scripting language and no MATLAB-related overheads have been taken under consideration. Blendenpik requires the FFTW library\footnote{http://www.fftw.org/}; we used FFTW-3.3.3. To match the accuracy of LAPACK's direct solvers, we fixed $\eps$ in Algorithm~\ref{alg:REK} to be 10e-14. Moreover, during our experiments we ensured that the residual error of all the competing algorithms were about of the same order of magnitude.

We used a Pentium(R) Dual-Core E5300 (2.60GHz) equipped with 5GB of RAM and compiled our source code using GCC-4.7.2 under Linux operating system. All running times displayed below are measured using the \emph{ftime} Linux system call by taking the average of the running time of 10 independent executions.

We experimented our algorithm under three different distributions of random input matrices (sparse, dense and ill-conditioned) under the setting of strongly rectangular settings of least squares instances. In all cases we normalized the column norms of the input matrices to unity and generate the right hand side vector $\b$ having Gaussian entries of variance one.

\paragraph{Sparse least squares}
We tested our algorithm in the overdetermined setting of random sparse $m\times n$ matrices with $n=800$ and $m=2000,3000, \ldots , 20000$ and density $0.25$. We also tested REK-BLAS on the underdetermined case where $m=800$ and $n=2000,3000, \ldots , 20000$. In both cases, the density of the sparse matrices was set to $0.25$ (for even sparser matrices REK-BLAS performed even better compared to all other mentioned methods). To generate these sparse matrix ensembles, we used MATLAB's \emph{sprandn} function with variance one. The results are depicted in Figure~\ref{fig:sparse}. Both plots demonstrate that REK-BLAS is superior on both the underdetermined (Figure~\ref{fig:sparseUnder}) and overdetermined case (Figure~\ref{fig:sparseOver}). It is interesting that REK-BLAS performs well in the underdetermined case.
\paragraph{Dense and well-conditioned least squares}
In this scenario, we used random overdetermined dense $m\times n$ matrices with much more rows than columns, i.e., we set $n=500$ and $m=1000,2000,\ldots ,20000$. We also tested REK-BLAS on the underdetermined case where $m=500$ and $n=1000,2000, \ldots, 20000$. We generated this set of matrices using MATLAB's \emph{randn} function with variance ten. We depicted the results in Figure~\ref{fig:denseOver}. In the overdetermined case (Figure~\ref{fig:denseOver}), REK-BLAS is marginally superior compared to LAPACK's routines whereas REK-C (as a naive implementation of Algorithm~\ref{alg:REK}) is inferior. Blendepik is the winner in this case. Interestingly, REK-BLAS almost matches the performance of Blendenpik in the underdetermined case, see Figure~\ref{fig:denseUnder}.
\begin{figure}[ht]
       \centering
                \includegraphics[width=0.85\textwidth]{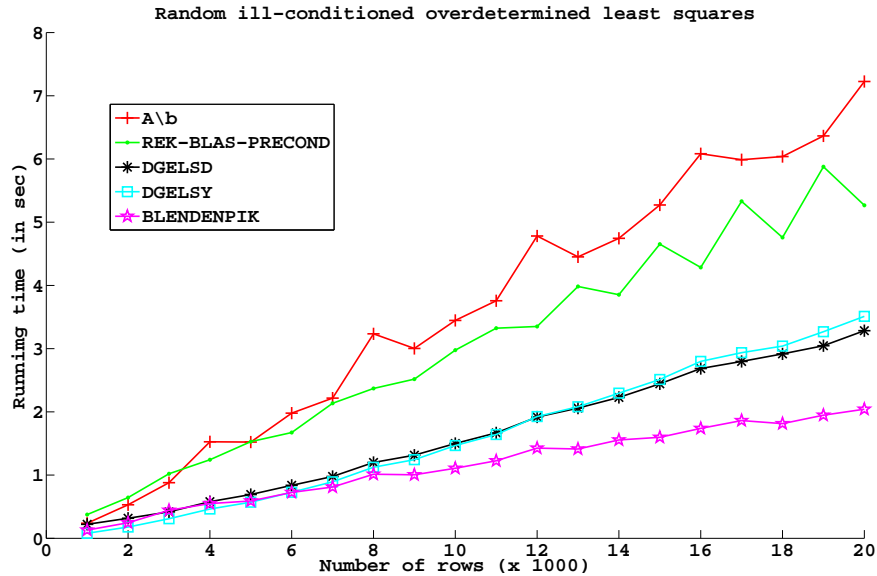}
                 \caption{Running time (in seconds) vs number of rows (scaled by 1000) for the case of random dense and ill-conditioned input matrices having 500 columns and condition number 10e6.}
                \label{fig:cond}
\end{figure}
\paragraph{Dense and ill-conditioned least squares}
Finally, we tested all algorithms under a particular case of random ill-conditioned dense matrices with $n=500$ and $m=1000,2000, \ldots , 20000$. Namely, we used Higham's \emph{randSVD} function for generating these matrices~\cite{testMat:Higham,book:higham}. More precisely, we set the condition number of these matrices to be $10e6$; set the top singular value to one and the rest to 10e-6. The results are displayed in Figure~\ref{fig:cond}. Unfortunately, in the ill-conditioned setting REK-BLAS-PRECOND is inferior compared to LAPACK's routines and Blendepik. We also verified the results of~\cite{AMT10} that Blendepik is superior compared to LAPACK's least squares solvers in this setting.
%
%%%%%%%%%%%%%%%%%%%%%%%%%%%%%%%%%%%%%%%%%%%%%%%%%%%%%%%
%%%%%%%%%%%%%%%%%%%%%%%%%%%%%%%%%%%%%%%%%%%%%%%%%%%%%%%
\section{Acknowledgements}
%%%%%%%%%%%%%%%%%%%%%%%%%%%%%%%%%%%%%%%%%%%%%%%%%%%%%%%
%%%%%%%%%%%%%%%%%%%%%%%%%%%%%%%%%%%%%%%%%%%%%%%%%%%%%%%
%
We would like to thank the anonymous reviewers for their invaluable comments on an earlier draft of the present manuscript. The first author would like to thank Haim Avron for his technical support on several issues regarding Blendenpik and Philip A. Knight for sharing his unpublished manuscript~\cite{K:error}.
%
%%%%%%%%%%%%%%%%%%%%%%%%%%%%%%%%%%%%%%%%%%%%%%%%%%%
%\clearpage
%%%%%%%%%%%%%%%%%%%%%%%%%%%%%%%%%%%%%%%%%%%%%%%%%%%
\bibliographystyle{alpha}

\begin{thebibliography}{DMMS11}

\bibitem[ABD{\etalchar{+}}90]{LAPACK}
E.~Anderson, Z.~Bai, J.~Dongarra, A.~Greenbaum, A.~McKenney, J.~Du~Croz,
  S.~Hammerling, J.~Demmel, C.~Bischof, and D.~Sorensen.
\newblock {LAPACK: a portable linear algebra library for high-performance
  computers}.
\newblock In {\em Proceedings of the 1990 ACM/IEEE conference on
  Supercomputing}, Supercomputing '90, pages 2--11. IEEE Computer Society
  Press, 1990.

\bibitem[AMT10]{AMT10}
H.~Avron, P.~Maymounkov, and S.~Toledo.
\newblock {Blendenpik: Supercharging LAPACK's Least-squares Solver}.
\newblock {\em SIAM Journal on Scientific Computing}, 32(3):1217--1236, 2010.

\bibitem[Ans84]{K:rate:Ansorge}
R.~Ansorge.
\newblock {Connections between the {Cimmino}-method and the {Kaczmarz}-method
  for the Solution of Singular and Regular Systems of Equations}.
\newblock 33(3--4):367--375, September 1984.

\bibitem[BBC{\etalchar{+}}87]{book:templates}
R.~Barrett, M.~Berry, T.~F. Chan, J.~Demmel, J.~Donato, J.~Dongarra,
  V.~Eijkhout, R.~Pozo, C.~Romine, and H.~van~der Vorst.
\newblock {\em {Templates for the Solution of Linear Systems: Building Blocks
  for Iterative Methods}}.
\newblock Software, Environments, Tools. Society for Industrial and Applied
  Mathematics, 1987.

\bibitem[Bj96]{book:Bjork}
A.~Bj\"{o}rck.
\newblock {\em {Numerical Methods for Least Squares Problems}}.
\newblock Society for Industrial and Applied Mathematics, 1996.

\bibitem[CEG83]{K:relax:CEG83}
Y.~Censor, P.~Eggermont, and D.~Gordon.
\newblock {Strong Underrelaxation in {K}aczmarz's Method for Inconsistent
  Systems}.
\newblock {\em Numerische Mathematik}, 41:83--92, 1983.

\bibitem[Cen81]{K:apps}
Y.~Censor.
\newblock {Row-Action Methods for Huge and Sparse Systems and Their
  Applications}.
\newblock {\em SIAM Review}, 23(4):444--466, 1981.

\bibitem[CRT11]{ROP:CRT11}
E.~S. Coakley, V.~Rokhlin, and M.~Tygert.
\newblock {A Fast Randomized Algorithm for Orthogonal Projection}.
\newblock {\em SIAM J. Sci. Comput.}, 33(2):849--868, 2011.

\bibitem[CW09]{CW09}
K.~L. Clarkson and D.~P. Woodruff.
\newblock {Numerical Linear Algebra in the Streaming Model}.
\newblock In {\em Proceedings of the Symposium on Theory of Computing (STOC)},
  pages 205--214, 2009.

\bibitem[CW12]{ls:nnzA}
K.~L. Clarkson and D.~P. Woodruff.
\newblock {Low Rank Approximation and Regression in Input Sparsity Time}.
\newblock Available at~arXiv:1207.6365, July 2012.

\bibitem[CZ97]{book:Zenios}
Y.~Censor and S.~A. Zenios.
\newblock {\em {Parallel Optimization: Theory, Algorithms, and Applications}}.
\newblock Numerical Mathematics and Scientific Computation Series. Oxford
  University Press, 1997.

\bibitem[Dem88]{cond:Demmel}
J.~W. Demmel.
\newblock {The Probability that a Numerical Analysis Problem is Difficult}.
\newblock {\em Mathematics of Computation}, 50(182):pp. 449--480, 1988.

\bibitem[DMM06]{petrosSODA06}
P.~Drineas, M.~W. Mahoney, and S.~Muthukrishnan.
\newblock {Sampling Algorithms for $\ell_2$-regression and Applications}.
\newblock In {\em Proceedings of the ACM-SIAM Symposium on Discrete Algorithms
  (SODA)}, pages 1127--1136, 2006.

\bibitem[DMMS11]{fasterLS}
P.~Drineas, M.~W. Mahoney, S.~Muthukrishnan, and T.~Sarl\`{o}s.
\newblock {Faster Least Squares Approximation}.
\newblock {\em Numer. Math.}, 117(2):219--249, Feb 2011.

\bibitem[FCM{\etalchar{+}}92]{K:apps:CFM92}
H.~G. Feichtinger, C.~Cenker, M.~Mayer, H.~Steier, and T.~Strohmer.
\newblock {New Variants of the POCS Method using Affine Subspaces of Finite
  Codimension with Applications to Irregular Sampling}.
\newblock pages 299--310, 1992.

\bibitem[FZ12]{FZ12}
N.~M. Freris and A.~Zouzias.
\newblock {Fast Distributed Smoothing for Network Clock Synchronization}.
\newblock {\em {IEEE Conference on Decision and Control (CDC)}}, 2012.

\bibitem[Gal03]{book:Galantai}
A.~Gal{\'a}ntai.
\newblock {\em Projectors and Projection Methods}.
\newblock Advances in Mathematics. Springer, 2003.

\bibitem[GBH70]{ART}
R.~Gordon, R.~Bender, and G.~T. Herman.
\newblock {Algebraic Reconstruction Techniques (ART) for three-dimensional
  electron microscopy and X-ray photography}.
\newblock {\em Journal of Theoretical Biology}, 29(3):471 -- 481, 1970.

\bibitem[GL96]{book:GVL}
G.~H. Golub and C.~F.~Van Loan.
\newblock {\em {Matrix Computations}}.
\newblock {The Johns Hopkins University Press}, {Third} edition, 1996.

\bibitem[GV61]{Chebyshev}
G.~H. Golub and R.~S. Varga.
\newblock {Chebyshev Semi-iterative Methods, Successive Overrelaxation
  Iterative Methods, and Second order Richardson Iterative Methods}.
\newblock {\em Numerische Mathematik}, 3:157--168, 1961.

\bibitem[Her80]{book:K:apps:H80}
G.~T. Herman.
\newblock {\em {I}mage {R}econstruction from {P}rojections {T}he {F}undamentals
  of {C}omputerized {T}omography}.
\newblock {C}omputer {S}cience and {A}pplied {M}athematics. {N}ew {Y}ork etc.:
  {A}cademic {P}ress ({A} {S}ubsidiary of {H}arcourt {B}race {J}ovanovich,
  {P}ublishers). {X}{I}{V}, 1980.

\bibitem[Hig89]{testMat:Higham}
N.~J. Higham.
\newblock {A Collection of Test Matrices in MATLAB}.
\newblock Technical report, Ithaca, NY, USA, 1989.

\bibitem[Hig96]{book:higham}
Nicholas~J. Higham.
\newblock {\em Accuracy and Stability of Numerical Algorithms}.
\newblock Society for Industrial and Applied Mathematics, Philadelphia, PA,
  USA, first edition, 1996.

\bibitem[HM93]{RK:HM93}
G.~T. Herman and L.~B. Meyer.
\newblock {Algebraic Reconstruction Techniques Can Be Made Computationally
  Efficient}.
\newblock {\em IEEE Transactions on Medical Imaging}, 12(3):600--609, 1993.

\bibitem[HN90]{K:relax:Hanke90}
M.~Hanke and Wilhelm N.
\newblock {On the Acceleration of Kaczmarz's Method for Inconsistent Linear
  Systems}.
\newblock {\em Linear Algebra and its Applications}, 130(0):83 -- 98, 1990.

\bibitem[Kac37]{K}
S.~Kaczmarz.
\newblock {Angen\"{a}herte Auflösung von Systemen Linearer Gleichungen}.
\newblock {\em Bulletin International de l'Académie Polonaise des Sciences et
  des Lettres}, 35:355--357, 1937.

\bibitem[Kni93]{phdthesis:K:error}
P.~A. Knight.
\newblock {\em {Error Analysis of Stationary Iteration and Associated
  Problems}}.
\newblock {{PhD} in Mathematics}, Manchester University, 1993.

\bibitem[Kni96]{K:error}
P.~A. Knight.
\newblock {A Rounding Error Analysis of Row-Action Methods}.
\newblock Unpublished manuscript, May 1996.

\bibitem[LL10]{LS:RCD}
D.~Leventhal and A.~S. Lewis.
\newblock {Randomized Methods for Linear Constraints: Convergence Rates and
  Conditioning}.
\newblock {\em Math. Oper. Res.}, 35(3):641--654, 2010.

\bibitem[McC75]{K:rate:MC77}
S.~F. McCormick.
\newblock {An Iterative Procedure for the Solution of Constrained Nonlinear
  Equations with Application to Optimization Problems}.
\newblock {\em Numerische Mathematik}, 23:371--385, 1975.

\bibitem[MSM11]{lsrn}
X.~Meng, M.~A. Saunders, and M.~W. Mahoney.
\newblock {LSRN: A Parallel Iterative Solver for Strongly Over- and
  Under-Determined Systems}.
\newblock Available at~http://arxiv.org/abs/1109.5981, Sept 2011.

\bibitem[MZ11]{MZ11}
A.~Magen and A.~Zouzias.
\newblock {Low Rank Matrix-Valued Chernoff Bounds and Approximate Matrix
  Multiplication}.
\newblock In {\em Proceedings of the ACM-SIAM Symposium on Discrete Algorithms
  (SODA)}, pages 1422--1436, 2011.

\bibitem[Nat01]{book:K:apps:Nat01}
F.~Natterer.
\newblock {\em {The Mathematics of Computerized Tomography}}.
\newblock Society for Industrial and Applied Mathematics, 2001.

\bibitem[NDT09]{Nguyen09}
N.~H. Nguyen, T.~T. Do, and T.~D. Tran.
\newblock {A Fast and Efficient Algorithm for Low-rank Approximation of a
  Matrix}.
\newblock In {\em Proceedings of the Symposium on Theory of Computing (STOC)},
  pages 215--224, 2009.

\bibitem[Nee10]{Needell09}
D.~Needell.
\newblock {Randomized Kaczmarz Solver for Noisy Linear Systems}.
\newblock {\em Bit Numerical Mathematics}, 50(2):395--403, 2010.

\bibitem[Pop99]{popa}
C.~Popa.
\newblock {Characterization of the Solutions Set of Inconsistent Least-squares
  Problems by an Extended Kaczmarz Algorithm}.
\newblock {\em Journal of Applied Mathematics and Computing}, 6:51--64, 1999.

\bibitem[PS82]{PS82}
C.C. Paige and M.A. Saunders.
\newblock {LSQR: An algorithm for sparse linear equations and sparse least
  squares}.
\newblock {\em ACM Transactions on Mathematical Software (TOMS)}, 8(1):43--71,
  1982.

\bibitem[RT08]{RT08}
V.~Rokhlin and M.~Tygert.
\newblock {A Fast Randomized Algorithm for Overdetermined Linear Least-squares
  Regression}.
\newblock {\em Proceedings of the National Academy of Sciences},
  105(36):13212--13218, 2008.

\bibitem[Saa03]{book:saad}
Y.~Saad.
\newblock {\em {Iterative Methods for Sparse Linear Systems}}.
\newblock Society for Industrial and Applied Mathematics, 2nd edition, 2003.

\bibitem[Sar06]{Sarlos}
T.~Sarl\`{o}s.
\newblock {Improved Approximation Algorithms for Large Matrices via Random
  Projections}.
\newblock In {\em Proceedings of the Symposium on Foundations of Computer
  Science (FOCS)}, pages 143--152, 2006.

\bibitem[Smi02]{aliasMethod:C}
W.~D. Smith.
\newblock {How to Sample from a Probability Distribution}.
\newblock Technical report, Princeton, NJ, USA, 2002.

\bibitem[SV09]{RK}
T.~Strohmer and R.~Vershynin.
\newblock {A Randomized Kaczmarz Algorithm with Exponential Convergence}.
\newblock {\em Journal of Fourier Analysis and Applications}, 15(1):262--278,
  2009.

\bibitem[Tan71]{K:Tanabe}
K.~Tanabe.
\newblock {Projection Method for Solving a Singular System of Linear Equations
  and its Applications}.
\newblock {\em Numerische Mathematik}, 17:203--214, 1971.

\bibitem[Tom55]{K:tompk}
C.~Tompkins.
\newblock {Projection Methods in Calculation}.
\newblock In {\em Proc. 2nd Symposium of Linear Programming}, pages 425--448,
  Washington, DC, 1955.

\bibitem[Vos91]{random:Alias}
M.~D. Vose.
\newblock {A Linear Algorithm for Generating Random Numbers with a given
  Distribution}.
\newblock {\em IEEE Trans. Softw. Eng.}, 17(9):972--975, September 1991.

\bibitem[Wal77]{sampling:Walker}
A.~J. Walker.
\newblock {An Efficient Method for Generating Discrete Random Variables with
  General Distributions}.
\newblock {\em ACM Trans. Math. Softw.}, 3(3):253--256, September 1977.

\bibitem[WM67]{K:rate:WM67}
T.~Whitney and R.~Meany.
\newblock {Two Algorithms related to the Method of Steepest Descent}.
\newblock {\em SIAM Journal on Numerical Analysis}, 4(1):109--118, 1967.

\end{thebibliography}
\newcommand{\etalchar}[1]{$^{#1}$}

%%%%%%%%%%%%%%%%%%%%%%%%%%%%%%%%%%%%%%%%%%%%%%%%%%%
%%%%%%%%%%%%%%%%%%%%%%%%%%%%%%%%%%%%%%%%%%%%%%%%%%%
%%%%%%%%%%%%%%%%%%%%%%%%%%%%%%%%%%%%%%%%%%%%%%%%%%%
%
%

%%%%%%%%%%%%%%%%%%%%%%%%%%%%%%%%%%%%%%%%%%%%%%%%%%%%%%%
%%%%%%%%%%%%%%%%%%%%%%%%%%%%%%%%%%%%%%%%%%%%%%%%%%%%%%%
\section{Appendix}
%%%%%%%%%%%%%%%%%%%%%%%%%%%%%%%%%%%%%%%%%%%%%%%%%%%%%%%
%%%%%%%%%%%%%%%%%%%%%%%%%%%%%%%%%%%%%%%%%%%%%%%%%%%%%%%
%
We present the proof of known facts from previous works for completeness.

%
%%%%%%%%%%%%%%%%%%%%%%%%%%%%%%%%%%%%%%%%%%%%%%%%%%%
\begin{proof}(of Lemma~\ref{lem:ortho})
It suffices to show that $\ip{\x^{(k+1)} - \xls}{\x^{(k+1)} - \x^{(k)}} = 0$. For notational convenience, let $\alpha_i := \frac{b_i - \ip{\x^{(k)}}{\ar{i}}}{\norm{\ar{i}}^2}$ for every $i\in{[m]}$. Assume that $\x^{(k+1)} = \x^{(k)} + \alpha_{i_k} \ar{i_k}$ for some arbitrary $i_k\in [m]$. Then,
\begin{align*}
	\ip{\x^{(k+1)} - \xls}{\x^{(k+1)} - \x^{(k)} }  & =   \ip{\x^{(k+1)} - \xls}{ \alpha_{i_k} \ar{i_k}} \ = \ \alpha_{i_k}\left(\ip{\x^{(k+1)} }{ \ar{i_k}} -  b_{i_k}\right)
\end{align*}
using the definition of $\x^{(k+1)}$, and the fact that $\ip{\xls}{\ar{i_k}} = b_{i_k}$ since $\xls$ is a solution to $\matA\x=\b$. Now, by the definition of $\alpha_{i_k} $, $\ip{\x^{(k+1)}}{\ar{i_k}} = \ip{\x^{(k)}}{\ar{i_k}} + \alpha_{i_k} \norm{\ar{i_k}}^2 = \ip{\x^{(k)}}{\ar{i_k}} + b_{i_k} - \ip{\x^{(k)}}{\ar{i_k}} = b_{i_k}$.
\end{proof}
%%%%%%%%%%%%%%%%%%%%%%%%%%%%%%%%%%%%%%%%%%%%%%%%%%%
%

%
%%%%%%%%%%%%%%%%%%%%%%%%%%%%%%%%%%%%%%%%%%%%%%%%%%%
\begin{proof}(of Lemma~\ref{lem:avg})
In light of Lemma~\ref{lem:ortho}, it suffices to show that $\EE_{Z}\norm{\x^{(k+1)} - \x^{(k)}}^2 \geq \frac1{\kappaFS(\matA)} \norm{\x^{(k)} - \xls}^2$.
By the definition of $\x^{(k+1)}$, it follows
\begin{align*}
\EE_{Z}\norm{\x^{(k+1)} - \x^{(k)}}^2  &= \EE_{Z} \left[\left(\frac{b_Z - \ip{\x^{(k)}}{\ar{Z} }}{\norm{\ar{Z}}^2}\right)^2 \norm{\ar{Z}}^2\right] \ =\  \EE_{Z} \frac{\ip{\xls - \x^{(k)} }{\ar{Z}}^2}{\norm{\ar{Z}}^2} \\
 & =   \sum_{i=1}^{m} \frac{\ip{\xls - \x^{(k)}}{\ar{i}}^2}{\frobnorm{\matA}^2} = \frac{\norm{\matA (\xls - \x^{(k)})}^2}{\frobnorm{\matA}^2}.
\end{align*}
By hypothesis, $\x^{(k)}$ is in the row space of $\matA$ for any $k$ when $\x^{(0)}$ is; in addition,
the same is true for $\xls$ by the definition of pseudo-inverse~\cite{book:GVL}. Therefore,
$\norm{\matA(\xls - \x^{(k)})} \geq \sigma_{\min} \norm{\xls - \x^{(k)} }$.
\end{proof}
%%%%%%%%%%%%%%%%%%%%%%%%%%%%%%%%%%%%%%%%%%%%%%%%%%%

%
\begin{proof}(of Theorem~\ref{thm:RK:inconsistent})
As in~\cite{Needell09}, for any $i\in{[m]}$ define the affine hyper-planes:
\begin{align*}
\mathcal{H}_i &:= \{\x: \ip{\ar{i}}{\x} = y_i\}\\
\mathcal{H}_i^{w_i} &:= \{\x: \ip{\ar{i} }{\x} = y_i + w_i\}
\end{align*}
Assume for now that at the $k$-th iteration of the randomized Kaczmarz algorithm applied on $(\matA, \b)$, the $i$-th row is selected. Note that $\hat\x^{(k)}$ is the projection of $\hat\x^{(k-1)}$ on $\mathcal{H}_i^{w_i}$ by the definition of the randomized Kaczmarz algorithm on input $(\matA,\b)$. Let us denote the projection of $\hat\x^{(k-1)}$ on $\mathcal{H}_i$ by $\x^{(k)}$. The two affine hyper-planes $\mathcal{H}_i,\mathcal{H}_i^{w_i}$ are \emph{parallel} with common normal $\ar{i}$, so $\x^{(k)}$ is the projection of $\hat\x^{(k)}$ on $\mathcal{H}_i$ and the minimum distance between $\mathcal{H}_i$ and $\mathcal{H}_i^{w_i}$ equals $|w_i| / \norm{\ar{i}}$. In addition, $\x^{*}\in \mathcal{H}_i$ since $\ip{\x^{*}}{\ar{i}} = y_i$, therefore by orthogonality we get that
\begin{equation}\label{eq:1}
\norm{\hat\x^{(k)} - \x^{*}}^2 = \norm{\x^{(k)} - \x^{*}}^2 + \norm{\hat\x^{(k)} - \x^{(k)}}^2.
\end{equation}
Since $\x^{(k)}$ is the projection of $\hat\x^{(k-1)}$ onto $\mathcal{H}_i$ (that is to say, $\x^{(k)}$ is a randomized Kaczmarz step applied on input $(\matA, \y)$ where the $i$-th row is selected on the $k$-th iteration) and $\hat\x^{(k-1)}$ is in the row space of $\matA$, Lemma~\ref{lem:avg} tells us that
\begin{equation}\label{eq:2}
 \EE\norm{\x^{(k)} - \x^{*}}^2 \le \left(1 - \frac1{\kappaFS(\matA)}\right) \norm{\hat\x^{(k-1)} - \x^{*}}^2.
\end{equation}
Note that for given selected row $i$ we have $ \norm{\hat\x^{(k)} - \x^{(k)}}^2= \frac{w_i^2}{\norm{\ar{i}}^2}$;
by the distribution of selecting the rows of $\matA$ we have that 
\begin{equation}\label{eq:3}
\E\norm{\hat\x^{(k)} - \x^{(k)}}^2 = \sum_{i=1}^{m} q_i \frac{w_i^2}{\norm{\ar{i}}^2} = \frac{\norm{\w}^2}{\frobnorm{\matA}^2}.
\end{equation}
Inequality~\eqref{ineq:relaxRK} follows by taking expectation on both sides of Equation~\eqref{eq:1} and bounding its resulting right hand side using Equations~\eqref{eq:2} and~\eqref{eq:3}. Applying Inequality~\eqref{ineq:relaxRK} inductively, it follows that 
\[\EE \norm{\hat\x^{(k)} - \x^{*}}^2 \le \left(1-\frac1{\kappaFS(\matA)}\right)^k\norm{\x^{(0)}- \x^{*}}^2 + \frac{\norm{\w}^2}{\frobnorm{\matA}^2}\sum_{i=0}^{k} \left(1-\frac1{\kappaFS(\matA)}\right)^i,\]
where we used that $\x^{(0)}$ is in the row space of $\matA$. The latter sum is bounded above by $\sum_{i=0}^{\infty} \left(1-\frac1{\kappaFS(\matA)}\right)^i =\frobnorm{\matA}^2 / \sigma^2_{\min}$.
\end{proof}
%%%%%%%%%%%%%%%%%%%%%%%%%%%%%%%%%%%%%%%%%%%%%%%%%%%%%%%
%%%%%%%%%%%%%%%%%%%%%%%%%%%%%%%%%%%%%%%%%%%%%%%%%%%%%%%

\end{document}